\documentclass[reqno,10pt]{amsart}

\usepackage{amsmath, amsthm, amsfonts}
\usepackage{pdfsync}
\usepackage{parskip}
\usepackage{hyperref}
\usepackage{mathrsfs}  
\usepackage{mathtools}
\usepackage{mdframed}
\usepackage{tikz}
\numberwithin{equation}{section}
\newtheorem{theorem}{Theorem}[section]
\newtheorem{prop}[theorem]{Proposition}
\newtheorem{definition}[theorem]{Definition}

\newtheorem{cor}[theorem]{Corollary}
\newtheorem{example}[theorem]{Example}
\theoremstyle{remark}
\newtheorem{remark}[theorem]{Remark}

\def\M{\mathsf{M}}
\def\sM{\mathscr{M}}
\def\p{\mathsf{p}}
\def\d{{\sf d}}

\def\ev{\mathsf{C}}
\def\cH{\mathcal{H}}
\def\sH{\mathscr{H}}
\def\cE{\mathcal{E}}

\def\J{\mathbb{J}}

\def\R{\mathbb{R}}
\def\C{\mathbb{C}}
\def\H{\mathbb{H}^{m}}

\def\N{\mathbb{N}}
\def\Nz{\mathbb{N}_0}
\def\A{\mathfrak{A}}
\def\K{\frak{K}}

\def\bE{\mathbb{E}}

\def\L{\mathcal{L}}
\def\Lis{\mathcal{L}{\rm{is}}}
\def\F{\mathfrak{F}}

\def\bbF{\mathbb{F}}
\def\bW{\boldsymbol{W}}
\def\bu{\boldsymbol{u}}

\def\B{\mathbb{B}}

\def\Ok{\mathsf{O}_{\kappa}}

\def\vpk{\varphi_{\kappa}}
\def\psk{\psi_{\kappa}}

\def\Q{\mathsf{Q}^{m}}
\def\Qk{\mathsf{Q}^{m}_{\kappa}}

\def\cA{\mathcal{A}}
\def\sA{\mathscr{A}}

\def\Aw{\mathscr{A}_\omega}
\def\cB{\mathcal{B}}
\def\fB{\mathfrak{B}}

\def\Bw{\mathcal{B}_\omega}

\def\aw{\mathfrak{a}_\omega}

\def\kf{\psi^{\ast}_{\kappa}}

\def\Rp{{\rm{Re}}}
\def\id{{\rm{id}}}

\def\gd{{\rm{grad}}}

\def\div{{\rm div}}

\def\sg{{\rm{sign}}}
\def\Hom{{\text{Hom}}}

\def\RS{\mathcal{RS}}
\def\cHi{\mathcal{H}^\infty}
\def\S{\mathcal{S}}

\def\Ric{{\rm Ric}}

\begin{document}

\title[The Yamabe flow on incomplete manifolds]{The Yamabe flow on incomplete manifolds}

\author[Y. Shao]{Yuanzhen Shao}
\address{Department of Mathematics,
         Purdue University, 
         150 N. University Street, 
          West La\-fayet\-te, IN 47907-2067, USA}
\email{shao92@purdue.edu}


\begin{abstract}
This article is concerned with developing an analytic theory for second order nonlinear parabolic equations on singular manifolds.
Existence and uniqueness of solutions in an $L_p$-framework is established by maximal regularity tools. 
These techniques are applied to the Yamabe flow. 
It is proven that the Yamabe flow admits a unique local solution within a class of incomplete initial metrics. 
\end{abstract}
\maketitle

\section{\bf Introduction}

Nowadays, there is a rising interest in the study of differential operators on manifolds with singularities, which is motivated by a variety of applications from applied mathematics, geometry and topology. 
All the work is more or less related to the seminal paper
by V.A. Kondrat'ev \cite{Kon67}. 
Among the tremendous amount of literature on this topic, I would like to mention two lines of research on pseudo-differential calculus of differential operators of Fuchs type, which have been introduced independently by R.B.~Melrose \cite{Mel81, Mel93} and B.-W.~Schulze \cite{NazSalSchSte06, Sch94, Sch97, SchSei02}. 

One important direction of research is connected with the so-called $b$-calculus and its generalizations on manifolds with cylindrical ends. See \cite{Mel81, Mel93}. Many authors have been very active in this direction.

Research along another line, known as cone differential operators, has also been known for a long time. 
Operators in this line of research are modelled on conical manifolds. 
During the recent decade, many mathematicians have applied analytic tools like bounded imaginary powers, $\cHi$-calculus and $\mathscr{R}$-sectoriality, see Section~4.1 for precise definitions, to study the realizations of cone differential operators on Mellin-Sobolev spaces and  to investigate many interesting nonlinear parabolic problems on conical manifolds. See for instance \cite{CorSchSei03, RoiSch13, RoiSch14, RoiSch15}.
A comparison between the $b$-calculus and the cone algebra can be found in \cite{LauSei99}.


There has been more recent progress in understanding elliptic operators on manifolds with higher order singularities, e.g., manifolds with edge ends. The reader may refer to \cite{MarSchz10, Maz91, MazVer14, Sch97, SchSei02, SchzSei06} for more details.
The amount of research on pseudo-differential calculus of differential operators of Fuchs type is enormous, and thus it is literally impossible to list all the work.

Geometric evolution equations by now are a well-established tool in the analysis of geometric and topological problems,  and thus forms probably the most important class of differential equations on Riemannian manifolds. 
In this article, our focus will mainly be the Yamabe flow on incomplete manifolds. 
As an alternative approach to the Yamabe problem, R.~Hamilton introduced the  Yamabe flow, which asks whether a metric, driven by this flow converges conformally to one  metric with constant scalar curvature.
More precisely, on a Riemannian manifold $(\M,g_0)$, the Yamabe flow studies the conformal evolution of metrics driven by the following rule. 
\begin{equation}
\label{S1: YF1}
\partial_t g=-R_g g \quad \text{on }\M_T,\quad g(0)=g_0\quad \text{on }\M, 
\end{equation}
where $R_g$ is the scalar curvature with respect to the metric $g$, and $\M_T:=\M\times [0,T]$. Let $g=u^{\frac{4}{m-2}}g_0$. This flow is equivalent to the following scalar quasilinear parabolic equation.
\begin{equation}
\label{S1: YF2}
\partial_t u =u^{-\frac{4}{m-2}}\Delta_{g_0} u -c(m) u^{\frac{m-6}{m-2}}R_{g_0}\quad \text{on }\M_T,\quad u(0)={\bf 1}_\M\quad \text{on }\M. 
\end{equation}
On compact closed manifold  $(\M,g_0)$, the short time existence of the Yamabe flow is just a consequence of the positivity of the conformal factor $u$ and the compactness of $(\M,g_0)$. 
Nevertheless, the theory for the Yamabe flow on non-compact manifolds is far from being settled. 
Even local well-posedness is only established for restricted situations. 
Its difficulty can be observed from the fact that, losing the compactness of $(\M,g_0)$, equation~\eqref{S1: YF2} can exhibit degenerate and singular behaviors simultaneously. The investigation of the Yamabe flow on non-compact manifolds was initiated by L. Ma and Y. An in \cite{MaAn99}. 
Later, conditions on extending local in time solutions were explored in \cite{MaCheZhu12}. 
In \cite{MaAn99}, the authors showed that
for a complete closed non-compact Riemannian manifold $(\M,g_0)$ with Ricci curvature bounded from below and with a uniform bound on the scalar curvature in the sense that: 
$$\Ric_{g_0}\geq -K g_0,\quad |R_{g_0}|\leq C,$$ 
equation~\eqref{S1: YF1} has short time solution on $\M\times [0,T(g_0)]$ for some $T(g_0)>0$. If in addition $R_{g_0}\leq 0$, then this solution is global. Here $\Ric_{g_0}$ is the Ricci curvature tensor with respect to $g_0$.
The proof is based on the widely used technique consisting of exhausting $\M$ with a sequence of compact manifolds with boundary and studying the solutions to a a sequence of initial boundary value problems. Then uniform estimates of these solutions and their gradients are obtained by means of the maximum principle on manifolds with Ricci curvature bounded from below in the sense given above. 

The more challenging case is to study the Yamabe flow on incomplete manifolds. E. Bahuaud and B. Vertman started the research in this direction in \cite{BahVer14} by considering the Yamabe flow on a compact manifold with asymptotically simple edge singularities. Their proof for short time existence is based on a careful analysis of the mapping property of the heat operator between H\"older spaces defined with respect to an edge metric. Very recently, the same authors established the long time existence of the Yamabe flow on edge manifolds in \cite{BahVer16}.

Another line of work on the Yamabe flow on geodesically incomplete manifold is to extend the problem onto manifolds with boundary. S. Brendle in \cite{Brend02} commenced the exploration of existence and convergence results on compact manifolds with vanishing mean curvature $H_g$ on the boundary, i.e.
\begin{equation}
\label{S1: YFB1}
\partial_t g=-R_g g \quad \text{on }\M_T,\quad H_g=0 \quad \text{on }\partial\M_T, \quad g(0)=g_0\quad \text{on }\M. 
\end{equation}
Under the conformal change $g=u^{\frac{4}{m-2}}g_0$, $H_g$ is related to $H_{g_0}$ by
$$H_g =\frac{2}{m-2} u^{-\frac{m}{m-2}} \left( \frac{\partial u}{\partial \nu_0} +\frac{m-2}{2} H_{g_0}u \right), $$
where $\nu_0$ is the unit outward normal vector with respect to the metric $g_0$.
Then the boundary condition $H_g\equiv 0$ becomes the homogeneous Neumann condition $\frac{\partial u}{\partial \nu_0}=0. $

I feel it worthwhile mentioning the studies of the well-known Ricci flow 
on incomplete manifolds. 
So far the research in this line mainly focuses on dimension two, in which case the Ricci flow coincides with the Yamabe flow up to a constant multiple. 
A remarkable feature of the Ricci flow for incomplete metrics is the loss of uniqueness. 
One important new notion of solutions is the instantaneously complete solutions, which 	originally appeared in a paper by E.~DiBenedetto and D.J.~Diller\cite{DibDil96}, and was later simplified and generalized  by G. Giesen and P. Topping. To describe it in more detail, we look at the Ricci flow on a possibly incomplete manifold $(\M,g_0)$.
\begin{align}
\label{S1: Ricci}
\partial_t g  = -2\Ric_g \quad \text{on }\M_T,\quad g(0)=g_0 \quad \text{on }\M.
\end{align}
A solution $(\M,g(t))$ is called instantaneously complete solution if for all $t>0$ $(\M,g(t))$ is  complete.
Many aspects of instantaneously complete solutions have been investigated in a series of papers by G. Giesen and P. Topping \cite{GieTop10, GieTop11, GieTop13, Top10, Top15, Zhu13}.

Another line of research on the Ricci flow, which indicates the loss of uniqueness of solutions to \eqref{S1: Ricci} on incomplete manifolds, studies those solutions starting from an incomplete metric and remaining incomplete for all time. In \cite{IseMazSes13, MarRubSes15, Yin10}, the analysis of the Ricci flow on manifolds with conical singularities was carried out. Roughly speaking, a conical metric near an isolated conical point can be expressed in polar coordinates $(t,x)\in (0,1]\times S^1$ by 
$$g_0= dt^2 +t^2dx^2 +\kappa ,$$
where $\kappa$ is a 2-tensor decaying like $t^\tau$ with $\tau>1$. Note that the conical point $r=0$  stays at a finite distance, and thus $(\M,g_0)$ is incomplete. The existence of both instantaneously complete solutions and conical solutions indicates the nonuniqueness of the solutions to the Ricci flow. 
A survey on the Ricci flow for incomplete metrics up to the year 2011 can be found in \cite{IseMazSes11}.

In 2014, the author of this article established an existence and uniqueness result for the Yamabe flow on singular manifolds under some mild initial regularity assumption. 
The result \cite[Theorem~4.4]{Shao14} shows how to construct examples of the Yamabe flow starting at a metric with unbounded scalar curvature.
The work in the current article extends the result in \cite{Shao14} and shows that the Yamabe flow can start with a class of incomplete metrics. This class includes examples different from those that have been studied for the Yamabe flow and the Ricci flow so far; the results are established via a totally different approach from the previous literature.
As a first example, we look at the following theorem.
\begin{theorem}
\label{S1: thm}
Suppose that $(\sM,g_0)$ is a compact manifold with boundary $\partial\sM= \partial_0 \sM \sqcup \partial_1\sM$, and $\partial_i\sM$ are both open. Let $(\M,g_0):=(\sM\setminus \partial_0\sM, g_0|_\M)$. Then the Yamabe flow~\eqref{S1: YFB1} admits a local in time $L_p$ solution.
\end{theorem}
The proof for this theorem is given in Theorem~\ref{S5.1: YF-thm}, Proposition~\ref{S5.2: YF-prop} and Example~\ref{S5.2: YF-eg}(ii). The explicit definition of $L_p$ solutions can be found in Section~4.3.
The Riemannian manifold $(\M,g_0)$ defined in Theorem~\ref{S1: thm} is  incomplete. 
At first glance, the Yamabe flow in Theorem~\ref{S1: thm} can be viewed as a flow on manifolds with boundary but with no boundary condition on $\partial_0\sM$, and thus is seemingly ill-posed. But as we will show in Section~5, this flow is a well-posed problem. 
Indeed, considering this problem  on the conformally complete manifold $(\M,g_0/\rho^2)$ for some $\rho\in C^\infty(\M; (0,\infty))$, \eqref{S1: YF1} becomes a singular parabolic equation. The core of the analysis of \eqref{S1: YF1} in Theorem~\ref{S1: thm} is how to offset the singular behavior of the conformal factor $u$ with a proper adaptation of the function spaces. 

This paper is organized as follows. 

In Section~2 and Section~3, some preliminary materials are presented. I introduce the precise definitions and properties of singular manifolds and function spaces used in this paper.  

The theoretic basis of this paper is Section~4. Based on the results from \cite{Shao15} and a modification of T.X. Duong's method in \cite{Duong89}, we can prove that a class of second order singular elliptic operators defined in \eqref{S4: operator A} is $\mathscr{R}$-sectorial. 
By the  maximal $L_p$-regularity theory, cf. \cite{DenHiePru03}, this implies that the class of singular parabolic equations considered in this article admits a unique local solution in an $L_p$-framework. 

In Section~5, the theory developed in Section~4 is applied to the Yamabe flow to establish existence and  uniqueness  of local solutions with incomplete initial metrics. The main theorem of this article, Theorem~\ref{S5.1: YF-thm}, is presented in Section~5.1. In Section~5.2, explicit examples of incomplete manifolds satisfying the conditions in  Theorem~\ref{S5.1: YF-thm} are given. 



\textbf{Notations:} 
Given any topological set $U$, $\mathring{U}$ denotes the interior of $U$. 
\smallskip\\
For any two Banach spaces $X,Y$, $X\doteq Y$ means that they are equal in the sense of equivalent norms. The notation $\Lis(X,Y)$ stands for the set of all bounded linear isomorphisms from $X$ to $Y$.
\smallskip\\
Given any Banach space $X$ and  manifold $\mathscr{M}$,
let $\| \cdot \|_\infty$, $\| \cdot \|_{s,\infty}$, $\|\cdot \|_p$ and $\|\cdot \|_{s,p}$ denote the usual norm of the $X$-valued Banach spaces $BC(\mathscr{M};X)$($L_\infty(\mathscr{M};X)$), $BC^s(\mathscr{M};X)$, $L_p(\mathscr{M};X)$ and $W^s_p(\mathscr{M};X)$, respectively. If the letter $X$ is omitted in the definition of these space, e.g., $BC(\sM)$, it means the corresponding space is $\C$-valued.
\smallskip\\ 

\section{\bf Singular manifolds satisfying property $\sH_\lambda$}

\subsection{\bf Uniformly regular Riemannian manifolds and Singular manifolds}

In \cite{Ama13, AmaAr}, H.~Amann introduced the concepts of {\em uniformly regular Riemannian manifolds} and {\em singular manifolds}

Let $(\M,g)$ be a $C^\infty$-Riemannian manifold of dimension $m$ with or without boundary endowed with $g$ as its Riemannian metric such that its underlying topological space is separable. An atlas $\A:=(\Ok,\vpk)_{\kappa\in \K}$ for $\M$ is said to be normalized if 
\begin{align*}
\vpk(\Ok)=
\begin{cases}
\Q, \hspace*{1em}& \Ok\subset\mathring{\M},\\
\Q\cap\H, &\Ok\cap\partial\M \neq\emptyset,
\end{cases}
\end{align*}
where $\H$ is the closed half space $\bar{\R}^+ \times\R^{m-1}$ and $\Q$ is the unit cube at the origin in $\R^m$. We put $\Qk:=\vpk(\Ok)$ and  $\psk:=\vpk^{-1}$. 

The atlas $\A$ is said to have \emph{finite multiplicity} if there exists $K\in \N $ such that any intersection of more than $K$ coordinate patches is empty. Put
\begin{align*}
\mathfrak{N}(\kappa):=\{\tilde{\kappa}\in\K:\mathsf{O}_{\tilde{\kappa}}\cap\Ok\neq\emptyset \}.
\end{align*} 
The finite multiplicity of $\A$ and the separability of $\M$ imply that $\A$ is countable.

An atlas $\A$ is said to fulfil the \emph{uniformly shrinkable} condition, if it is normalized and there exists $r\in (0,1)$ such that $\{\psk(r{\Qk}):\kappa\in\K\}$ is a cover for ${\M}$. 

Following H.~Amann \cite{Ama13, AmaAr}, we say that $(\M,g)$ is a {\bf{uniformly regular Riemannian manifold}} if it admits an atlas $\A$ such that
\begin{itemize}
\item[(R1)] $\A$ is uniformly shrinkable and has finite multiplicity. If $\M$ is oriented, then $\A$ is orientation preserving.
\item[(R2)] $\|\varphi_{\eta}\circ\psk \|_{k,\infty}\leq c(k) $, $\kappa\in\K$, $\eta\in\mathfrak{N}(\kappa)$, and $k\in{\N}_0$.
\item[(R3)] $\kf g\sim g_m $, $\kappa\in\K$. Here $g_m$ denotes the Euclidean metric on ${\R}^m$ and $\kf g$ denotes the pull-back metric of $g$ by $\psk$.
\item[(R4)] $\|\kf g\|_{k,\infty}\leq c(k)$, $\kappa\in\K$ and $k\in\Nz$.
\end{itemize}
Here $\|u\|_{k,\infty}:=\max_{|\alpha|\leq k}\|\partial^{\alpha}u\|_{\infty}$, and it is understood that a constant $c(k)$, like in (R2), depends only on $k$. An atlas $\A$ satisfying (R1) and (R2) is called a \emph{uniformly regular atlas}. (R3) reads as
\begin{center}
$|\xi|^2/c\leq \kf g(x)(\xi,\xi) \leq{c|\xi|^2}$,\hspace{.5em} for any $x\in \Qk,\xi\in \R^m, \kappa\in\K$ and some $c\geq{1}$.
\end{center}
In \cite{DisShaoSim}, it is shown that the class of {\em uniformly regular Riemannian manifolds} coincides with the family of complete Riemannian manifolds with bounded geometry, when $\partial\M=\emptyset$.

Assume that $\rho\in C^{\infty}(\M,(0,\infty))$. Then $(\rho,\K)$ is a {\em singularity datum} for $\M$ if
\begin{itemize}
\item[(S1)] $(\M,g/\rho^2)$ is a {\em uniformly regular Riemannian manifold}.
\item[(S2)] $\A$ is a uniformly regular atlas.
\item[(S3)] $\|\kf\rho\|_{k,\infty}\leq c(k)\rho_{\kappa}$, $\kappa\in\K$ and $k\in\N_0$, where $\rho_{\kappa}:=\rho(\psk(0))$.
\item[(S4)] $\rho_{\kappa}/c\leq \rho(\p)\leq c\rho_{\kappa}$, $\p\in\Ok$ and $\kappa\in\K$ for some $c\geq 1$ independent of $\kappa$.
\end{itemize}
Two {\em singularity data} $(\rho,\K)$ and $(\tilde{\rho},\tilde{\K})$ are equivalent, if
\begin{itemize}
\item[(E1)] $\rho\sim \tilde{\rho}$.
\item[(E2)] card$\{\tilde{\kappa}\in\tilde{\K}:\mathsf{O}_{\tilde{\kappa}}\cap\Ok\neq\emptyset\}\leq c$, $\kappa\in\K$.
\item[(E3)] $\|\varphi_{\tilde{\kappa}}\circ\psk\|_{k,\infty}\leq{c(k)}$, $\kappa\in\K$, $\tilde{\kappa}\in\tilde{\K}$ and $k\in{\N}_0$
\end{itemize}
We write the equivalence relationship as $(\rho,\K)\sim(\tilde{\rho},\tilde{\K})$. (S1) and (E1) imply that 
\begin{align}
\label{section 1:singular data}
1/c\leq \rho_{\kappa}/\tilde{\rho}_{\tilde{\kappa}}\leq c,\hspace*{.5em} \kappa\in\K,\hspace*{.5em} \tilde{\kappa}\in\tilde{\K}\text{ and }\mathsf{O}_{\tilde{\kappa}}\cap\Ok\neq \emptyset.
\end{align}
{\em A singularity structure}, $\mathfrak{S}(\M)$, for $\M$ is a maximal family of equivalent {\em singularity data}. A {\em singularity function} for $\mathfrak{S}(\M)$ is a function $\rho\in C^{\infty}(\M,(0,\infty))$ such that there exists an atlas $\A$ with $(\rho,\A)\in\mathfrak{S}(\M)$. The set of all {\em singularity functions} for $\mathfrak{S}(\M)$ is the {\em singular type}, $\mathfrak{T}(\M)$, for $\mathfrak{S}(\M)$. By a {\bf{singular manifold}} we mean a Riemannian manifold $\M$ endowed with a singularity structure $\mathfrak{S}(\M)$. Then $\M$ is said to be \emph{singular of type} $\mathfrak{T}(\M)$. If $\rho\in\mathfrak{T}(\M)$, then it is convenient to set $[\![\rho]\!]:=\mathfrak{T}(\M)$ and to say that $(\M,g;\rho)$ is a {\em singular manifold}. A {\em singular manifold} is a {\em uniformly regular Riemannian manifold} iff $\rho\sim {\bf 1}_{\M}$. 

We refer to \cite{Ama13b, Ama14} for examples of {\em uniformly regular Riemannian manifolds} and {\em singular manifolds}.


Lastly, for each $k\in\N$, the concept of {\bf{$C^k$-uniformly regular Riemannian manifold}} is defined by modifying (R2), (R4) and (L1), (L2) in an obvious way. Similarly, {\bf{$C^k$-singular manifolds}} are defined by replacing the smoothness of $\rho$ by $\rho\in C^k(\M,(0,\infty))$ and altering (S1)-(S3) accordingly.

\subsection{\bf Singular manifolds satisfying property $\sH_\lambda$}

In this subsection, a class of {\em singular manifolds}, called {\em singular manifolds} satisfying property $\sH_\lambda$, is introduced. This concept has proven itself useful for the theory of second order differential equations on {\em singular manifolds} in \cite{Shao15}.
 
\begin{definition}
A function $h\in C^2(\M;\R)$ is said to belong to the class $\sH_\lambda(\M,g;\rho)$ with parameters $(c,M)$ for some $c\geq 1$ and $M>0$, if
\begin{itemize}
\item[($\sH_\lambda$1)] $M/c \leq \rho |\gd h|_g \leq Mc$;
\item[($\sH_\lambda$2)] $M/c \leq \rho^\lambda \div(\rho^{2-\lambda}\gd h) \leq Mc.$
\end{itemize}
\end{definition}
The precise definition of the divergence operator $\div$ is given in \eqref{S2: divergence}.
Observe that if $h\in \sH_\lambda(\M,g;\rho)$ with parameters $(c,1)$, then $Mh\in \sH_\lambda(\M,g;\rho)$ with parameters $(c,M)$. 

\begin{definition}
A  singular manifold $(\M,g;\rho)$ is said to enjoy property $\sH_\lambda$, if there exists some $h\in \sH_\lambda(\M,g;\rho)$. 
\end{definition}

Examples of {\em singular manifolds} with property $\sH_\lambda$ can be found in \cite[Section~5.2]{Shao15}. 
We will introduce one example, which will be of particular interest in the following sections.

Let $J_0:=(0,1]$. 
We denote by $\mathscr{R}(J_0)$ the set of all $R\in C^\infty(J_0,(0,\infty))$ with $R(1)=1$ such that 
$R(0):=\lim\limits_{t\to 0} R(t)$ exists in $[0,\infty]$. We write $R\in \mathscr{C_U}(J_0)$ if
\begin{align}
\label{S4: unif mild cusp ch}
\begin{cases}
\text{(i)} \quad & R\in \mathscr{R}(J_0);\\
\text{(ii)}  & \int_{J_0} dt/R(t)=\infty;\\
\text{(iii)}  & 	\|\partial^k_t R\|_{\infty} <\infty,\quad k\geq 1;\\
\text{(iv)} & \partial_t R \sim {\bf 1}_{J_0},\quad |\partial^2_t R| <\infty.
\end{cases}
\end{align}
Then we call $R$ a {\em uniformly mild cusp characteristic}. The first three conditions were introduced by H. Amann in \cite{Ama14} to define {\em  cusp characteristics}.

Assume that $(B,g_B)$ is a $d$-dimensional uniformly regular Riemannian submanifold without boundary of $\R^{\bar{d}}$, and $R\in \mathscr{C_U}(J_0)$. The (model) $(R,B)$-cusp $P(R,B)$ on $J_0$, also called $R$-cusp over $B$ on $J_0$, is defined by 
$$P(R,B)=P(R,B;J_0):=\{(t,R(t)y):\, t\in J_0, \, y\in B\}\subset \R^{1+\bar{d}} .$$
It is a $(1+d)$-dimensional submanifold of $\R^{1+\bar{d}}$. 
The map
$$\phi_P=\phi_P(R):P \rightarrow J_0\times B: \quad (t,R(t)y)\mapsto  (t,y) $$
is a diffeomorphism, the {\em canonical stretching diffeomorphism} of $P$. 

Assume that $(\Gamma, g_\Gamma)$ is a compact connected Riemannian manifold without boundary. Then the (model) $\Gamma$-wedge over the $(R,B)$-cusp, $P(R,B)$, is defined by
$$W=W(R,B,\Gamma):=P(R,B)\times\Gamma.$$
If $\Gamma$ is a one-point space, then $W$ is naturally identified with $P$. Thus every cusp is also a wedge.

We call $(\M,g)$ a {\em closed uniformly mild $\Gamma$-wedge over $P(R,B)$},  if there is a diffeomorphism $f: \M \to W(R,B,\Gamma)$ such that $g=f^*( \phi_P^*(dt^2+g_B) +g_\Gamma)$.
\begin{example}
\label{S2.2: Eg-mild cusp-wedge}
Let $\lambda\in [0,1)\cup(1,\infty)$. Assume that $(\M,g)$ is a closed uniformly mild $\Gamma$-wedge over $P(R,B)$. 
Then $(\M,g)$ is a singular manifold  satisfying property $\sH_\lambda$-end.
\end{example}
\begin{proof}
See \cite[Remark~5.11]{Shao15}.
\end{proof}

\section{\bf Weighted function spaces on singular manifolds}

In this Section, I follow the work of H. Amann in \cite{Ama13} and \cite{AmaAr} to introduce some concepts and properties of weighted function spaces on {\em singular manifolds}.

\subsection{\bf Tensor bundles}
Suppose $(\M,g;\rho)$ is a {\em singular manifold}.
Given $\sigma,\tau\in\N_0$, 
$$T^{\sigma}_{\tau}{\M}:=T{\M}^{\otimes{\sigma}}\otimes{T^{\ast}{\M}^{\otimes{\tau}}}$$ 
is the $(\sigma,\tau)$-tensor bundle of $\M$, where $T{\M}$ and $T^{\ast}{\M}$ are the (complexified) tangent and the cotangent bundle of ${\M}$, respectively.
We write $\mathcal{T}^{\sigma}_{\tau}{\M}$ for the $C^{\infty}({\M})$-module of all smooth sections of $T^{\sigma}_{\tau}\M$,
and $\Gamma(\M,T^{\sigma}_{\tau}{\M})$ for the set of all sections.

For abbreviation, we set $\J^{\sigma}:=\{1,2,\ldots,m\}^{\sigma}$, and $\J^{\tau}$ is defined alike. Given local coordinates $\varphi=\{x^1,\ldots,x^m\}$, $(i):=(i_1,\ldots,i_{\sigma})\in\J^{\sigma}$ and $(j):=(j_1,\ldots,j_{\tau})\in\J^{\tau}$, we set
\begin{align*}
\frac{\partial}{\partial{x}^{(i)}}:=\frac{\partial}{\partial{x^{i_1}}}\otimes\cdots\otimes\frac{\partial}{\partial{x^{i_{\sigma}}}}, \quad
 dx^{(j)}:=dx^{j_1}\otimes{\cdots}\otimes{dx}^{j_{\tau}}.
\end{align*}
The local representation of 
$a\in \Gamma(\M,T^{\sigma}_{\tau}{\M})$ with respect to these coordinates is given by 
\begin{align}
\label{local}
a=a^{(i)}_{(j)} \frac{\partial}{\partial{x}^{(i)}} \otimes dx^{(j)} 
\end{align}
with coefficients $a^{(i)}_{(j)}$ defined on $\Ok$。
\smallskip\\
We denote by $\nabla=\nabla_g$ the extension of the Levi-Civita connection over $\mathcal{T}^{\sigma}_{\tau}{\M}$.
$\langle{\cdot,\cdot}\rangle:\mathcal{T}^{\sigma}_{\tau}{\M}\times{\mathcal{T}^{\tau}_{\sigma}{\M}}\rightarrow{C^{\infty}({\M})}$ stands for the extension of the fiber-wise defined duality pairing on ${\M}$, cf. \cite[Section 3]{Ama13}.
Set $\nabla_{i}:=\nabla_{\partial_{i}}$ with $\partial_{i}=\frac{\partial}{\partial{x^i}}$. 
For $k\in\Nz$, we define
$$\nabla^k: \mathcal{T}^{\sigma}_{\tau}{\M}\rightarrow{\mathcal{T}^{\sigma}_{\tau+k}{\M}},\quad a\mapsto{\nabla^k{a}}$$
by letting $\nabla^0 a:=a$ and $\nabla^{k+1} a:=\nabla\circ\nabla^k a$.
We can also extend the Riemannian metric $(\cdot|\cdot)_g$ from the tangent bundle to any $(\sigma,\tau)$-tensor bundle $T^{\sigma}_{\tau}{\M}$, which is still written as $(\cdot|\cdot)_g$. Meanwhile, $(\cdot|\cdot)_{g^*}$ stands for the induced contravariant metric. 
In addition,
$$|\cdot|_g:=|\cdot|_{g^\tau_\sigma}:\mathcal{T}^{\sigma}_{\tau}{\M}\rightarrow{C^{\infty}}({\M}),\quad a\mapsto\sqrt{(a|a)_g}$$
is called the (vector bundle) \emph{norm} induced by $g$.
\smallskip\\
We assume that $V$ is a $\C$-valued tensor bundle on $\M$, i.e.,
$$V=V^{\sigma}_{\tau}:=\{T^{\sigma}_{\tau}\M, (\cdot|\cdot)_g\},$$ 
for some $\sigma,\tau\in\N_0$. 

Throughout the rest of this paper, unless stated otherwise, we always assume that 
\smallskip
\begin{mdframed}
\begin{itemize}
\item $(\M,g;\rho)$ is a {\em singular manifold}.
\item $\rho\in \mathfrak{T}(\M)$, $s\geq 0$, $k\in\Nz$, $1<p<\infty$ and $\vartheta\in \R$.
\item $\sigma,\tau\in \Nz$, $V=V^{\sigma}_{\tau}:=\{T^{\sigma}_{\tau}\M, (\cdot|\cdot)_g\}$.
\end{itemize}
\end{mdframed}


\subsection{\bf Weighted function spaces}

We denote by $\mathcal{D}({\M},V)$ the space of smooth sections of $V$ that is compactly supported in $\M$.
Then the weighted Sobolev space $W^{k,\vartheta}_p({\M},V)$ is defined as the completion of $\mathcal{D}({\M},V)$ in $L_{1,loc}(\M,V)$ with respect to the norm
\begin{center}
$\|\cdot\|_{k,p;\vartheta}: u\mapsto(\sum_{i=0}^{k}\|\rho^{\vartheta+i+\tau-\sigma}|\nabla^{i}u|_{g}\|_p^p)^{\frac{1}{p}}$.
\end{center}
Note that $W^{0,\vartheta}_p(\M, V)=L^{\vartheta}_p(\M, V)$ with equal norms. In particular, we can define the weighted spaces $L_q^\vartheta(\M,V)$ for $q\in \{1,\infty\}$ in a similar manner.

Define
$$BC^{k,\vartheta}(\M,V):=(\{u\in{C^k({\M},V)}:\|u\|_{k,\infty;\vartheta}<\infty\},\|\cdot\|_{k,\infty;\vartheta}),$$
where $\|u\|_{k,\infty;\vartheta}:={\max}_{0\leq{i}\leq{k}}\|\rho^{\vartheta+i+\tau-\sigma}|\nabla^{i}u|_{g}\|_{\infty}$.
We also set
$$BC^{\infty,\vartheta}({\M},V):=\bigcap_{k}BC^{k,\vartheta}(\M, V).$$ 
The weighted Sobolev-Slobodeckii spaces are defined as
\begin{align}
\label{S2: fractional Sobolev}
W^{s,\vartheta}_p({\M},V):=(L^{\vartheta}_p(\M, V),W^{k,\vartheta}_p(\M, V))_{s/k,p},
\end{align}
for $ s\in \R_+\setminus\Nz$, $k=[s]+1$,
where $(\cdot,\cdot)_{\theta,p}$ is the real interpolation method \cite[Chapter I.2.4.1]{Ama95}.

In the special case that $(\M,g)$ is uniformly regular, since $\rho\sim {\bf 1}_\M$, the definition of any weighted space $\F^{s,\vartheta}(\M,V)$ is actually independent of the weight $\vartheta$. In this case, all spaces are indeed unweighted. We thus denote these spaces simply by $\F^s(\M,V)$.

In the rest of this subsection, several properties of weighted function spaces are presented without proof. Their proofs can be found in \cite{Shao15}.

Let $V_j=V^{\sigma_j}_{\tau_j}:=\{T^{\sigma_j}_{\tau_j}\M,(\cdot|\cdot)_g\}$ with $j=1,2,3$ be $\C$-valued tensor bundles on $\M$. By bundle multiplication from $V_1\times V_2$ into $V_3$, denoted by
\begin{center}
${\mathsf{m}}: V_1\times V_2\rightarrow V_3$,\hspace{1em} $(v_1,v_2)\mapsto {\mathsf{m}}(v_1,v_2)$,
\end{center}
we mean a smooth bounded section $\mathfrak{m}$ of $\Hom(V_1\otimes V_2,V_3)$, i.e., 
\begin{align}
\label{section 2: bundle multiplication}
\mathfrak{m}\in BC^{\infty}(\M, \text{Hom}(V_1\otimes V_2,V_3)), 
\end{align}
such that $\mathsf{m}(v_1,v_2):=\mathfrak{m}(v_1\otimes v_2)$. \eqref{section 2: bundle multiplication} implies that  for some $c>0$
\begin{center}
$|{\mathsf{m}}(v_1,v_2)|_g \leq c|v_1|_g |v_2|_g$,\hspace{1em} $v_i\in \Gamma(\M,V_i)$ with $i=1,2$.
\end{center}
Its point-wise extension from $\Gamma(\M,V_1\oplus V_2)$ into $\Gamma(\M,V_3)$ is defined by:
\begin{align*}
\mathsf{m}(v_1,v_2)(\p):=\mathsf{m}(\p)(v_1(\p),v_2(\p))
\end{align*}
for $v_i\in \Gamma(\M,V_i)$ and $p\in\M$. We still denote it by ${\mathsf{m}}$. We can formulate the following point-wise multiplier theorem for function spaces over {\em singular manifolds}.

\begin{prop}
\label{S2: pointwise multiplication}
Let $k\in\Nz$. Assume that the tensor bundles $V_j=V^{\sigma_j}_{\tau_j}:=\{T^{\sigma_j}_{\tau_j}\M,(\cdot|\cdot)_g\}$ with $j=1,2,3$ satisfy 
\begin{align}
\label{section 2: ptm-condition}
\sigma_3-\tau_3=\sigma_1+\sigma_2-\tau_1-\tau_2.
\end{align}
Suppose that $\mathsf{m}:V_1\times V_2\rightarrow V_3$ is a bundle multiplication, and $\vartheta_3=\vartheta_1+\vartheta_2$. Then $[(v_1,v_2)\mapsto \mathsf{m}(v_1,v_2)]$ is a bilinear and continuous map for $k\in\Nz$ and $s\leq k$
$$BC^{k,\vartheta_1}(\M, V_1)\times W^{s,\vartheta_2}_p(\M, V_2)\rightarrow W^{s,\vartheta_3}_p(\M, V_3).$$ 
\end{prop}

\begin{prop}
\label{S2: nabla}
For $\F\in \{BC, W_p\}$, we have
$$\nabla\in \L(\F^{s+1,\vartheta}(\M, V^\sigma_\tau), \F^{s,\vartheta}(\M, V^\sigma_{\tau+1}),$$
and
$$\gd \in \L(\F^{s+1,\vartheta}(\M, V^\sigma_\tau), \F^{s,\vartheta+2}(\M, V^{\sigma+1}_\tau)).$$
\end{prop}

\begin{prop}
\label{S2: change of wgt}
For $\F\in\{BC, W_p\}$, we have
$$f_{\vartheta}:=[u\mapsto \rho^{\vartheta}u] \in \Lis(\F^{s,\vartheta^\prime+\vartheta}(\M, V),\F^{s,\vartheta^\prime}(\M, V)).$$ 
\end{prop}


We denote by $\ev^{\sigma+1}_{\tau+1}: V^{\sigma+1}_{\tau+1}\rightarrow V^\sigma_\tau$ the contraction with respect to position $\sigma+1$ and $\tau+1$, that is for any $(i)\in\J^\sigma$, $(j)\in\J^\tau$ and $k,l\in\J^1$ and $\p\in\M$
$$
\ev^{\sigma+1}_{\tau+1}a  :=\ev^{\sigma+1}_{\tau+1} a^{(i;k)}_{(j;l)} \frac{\partial}{\partial x^{(i)}}\otimes \frac{\partial}{\partial x^k}\otimes dx^{(j)}\otimes dx^l:=a^{(i;k)}_{(j;k)} \frac{\partial}{\partial x^{(i)}} \otimes dx^{(j)} 
$$
in every local chart. Recall that the surface divergence of tensor fields with respect to the metric $g$ is the map
\begin{equation}
\label{S2: divergence}
\div=\div_g: C^1(\M,V^{\sigma+1}_\tau)\rightarrow C(\M, V^\sigma_\tau), \quad a\mapsto \ev^{\sigma+1}_{\tau+1}(\nabla a).
\end{equation}

Put $V^\prime:=V^\tau_\sigma$. Let $\ev: V^{\sigma+\sigma_1}_{\tau+\tau_1} \times V^\prime \rightarrow V^{\sigma_1}_{\tau_1}$ denote the complete contraction. For any $a\in V^{\sigma+\sigma_1}_{\tau+\tau_1}$ and $b\in V^\prime$, the complete contraction (on the right)  is defined by
\begin{align*}
\ev(a,b)= a^{(i;i_1)}_{(j;j_1)} b^{(j)}_{(i)} \frac{\partial}{\partial x^{(i_1)}}\otimes dx^{(j_1)},
\end{align*}
with $(i)\in\mathbb{J}^\sigma$ ,$(i_1)\in\mathbb{J}^{\sigma_1}$, $(j) \in\mathbb{J}^\tau$,$(j_1)\in\mathbb{J}^{\tau_1}$, in local coordinates. 
The complete contraction (on the left) is defined in an analogous manner. Note that the complete contraction is a bundle multiplication.

Suppose that $\partial\M\neq \emptyset$. Since $T(\partial\M)$ is a subbundle of codimension $1$ of the vector bundle $(T\M)_{\partial\M}$ over $\partial\M$, there exists a unique vector field $\boldsymbol{n}$ in $(T\M)_{\partial\M}$ of length $1$ orthogonal to $T(\partial\M)$, and inward pointing. In every local coordinate system, $\varphi_\kappa=\{x_1,\cdots,x_m\}$
$$\boldsymbol{n}=\frac{1}{\sqrt{g_{11}|\partial\Ok}} \frac{\partial}{\partial x^1}. $$
We consider the boundary operator
$$
\cB_i u =
\begin{cases}
\gamma u \quad &\text{on }\partial \M \quad \text{when $i$=0},\\
\ev(\nabla u,\boldsymbol{n}) &\text{on }\partial \M \quad \text{when $i$=1}.
\end{cases}
$$
Here $\gamma$ is the trace operator on $\partial\M$. We define for $s\neq i+1/p$ the space
\begin{equation}
\label{S3: Sobolev space-boundary}
W^{s,\vartheta}_{p,\cB_i}(\M,V):=\{u\in W^{s,\vartheta}_p(\M,V): \cB_i u=0\}.
\end{equation}
Note that this is a closed subspace of $W^{s,\vartheta}_p(\M, V)$.
In particular,
$$W^{s,\vartheta}_{p,\cB_i}(\M,V)=W^{s,\vartheta}_p(\M, V),\quad 0\leq s <i+1/p.$$ 
See \cite[Section~21]{AmaAr}. 


\begin{prop}
\label{S2: div-tangent}
For  $\F\in \{BC, W_p\}$, it holds
$$\div\in\L(\F^{s+1,\vartheta}(\M, V^{\sigma+1}_\tau), \F^{s,\vartheta}(\M, V^\sigma_\tau)).$$
\end{prop}

The following results are not contained in \cite{Shao15}.
\begin{theorem}
\label{S2: divergence thm}
For any $a\in W^{1,-\vartheta}_2(\M,V^\prime)$  and $b\in W^{1,\vartheta}_2(\M,V^{\sigma+1}_{\tau})$
$$\int\limits_{\partial\M} \langle\ev(b_\flat,\boldsymbol{n}), a \rangle \, d\mu_g  - \int\limits_\M \langle \div b, a \rangle\, d\mu_g= \int\limits_\M \langle b, \nabla a \rangle\, d\mu_g.$$
\end{theorem}
\begin{proof}
This follows from the divergence theorem and the proof for \cite[Theorem~2.6]{Shao15}.
\end{proof}

\begin{prop}
\label{S2: interpolation}
Suppose that $k_i\in \Nz$, $\vartheta_i\in \R$ with $i=0,1$, $0<\theta<1$ and $k_0<k_1$ . Then
$$(W^{k_0,\vartheta_0}_{p,\cB_i}(\M,V), W^{k_1,\vartheta_1}_{p,\cB_i}(\M,V))_{\theta,p} \doteq W^{k_\theta,\vartheta_\theta}_{p,\cB_i}(\M,V).$$
Here $\xi_\theta:=(1-\theta)\xi_0+\theta \xi_1$ for any $\xi_0,\xi_1\in \R$. When $\partial\M=\emptyset$, $k_\theta\notin \N$; while $\partial\M\neq \emptyset$, $k_\theta\neq i+1/p$.
\end{prop}
\begin{proof}
This follows from \cite[Theorem~8.2(i), formula~(8.3), Theorem~21.3]{AmaAr} and the proof of \cite[formula~(21.2)]{AmaAr}
\end{proof}

\begin{prop}
\label{S2: Sobolev embedding}
Suppose that $\vartheta\in\R$ and $s>k+\frac{m}{p}$. Then
$$W^{s,\vartheta}_p(\M) \hookrightarrow BC^{k,\vartheta+\frac{m}{p}}(\M). $$
\end{prop}
\begin{proof}
See \cite[Theorem~14.2(ii)]{Ama13}.
\end{proof}

\begin{prop}
\label{S2: embedding}
Suppose that $\vartheta_1<\vartheta_0$ and $\F\in \{BC, W_p,  W_{p,\cB_i}\}$. Then
\begin{align*}
\begin{cases}
\F^{s,\vartheta_1}(\M) \hookrightarrow \F^{s,\vartheta_0}(\M)\quad &\text{if }\rho\leq 1;\\
\F^{s,\vartheta_0}(\M) \hookrightarrow \F^{s,\vartheta_1}(\M) &\text{if } \rho\geq 1. 
\end{cases}
\end{align*}
\end{prop}
\begin{proof}
See \cite[Theorem~15.1]{AmaAr}.
\end{proof}


\section{\bf Maximal $L_p$-regularity of  second order elliptic differential operators}

Suppose that $\sA: \mathcal{D}(\mathring{\M})\rightarrow \Gamma(\M)$ is a second order differential operator defined as follows.
\begin{align}
\label{S4: operator A}
\sA u:= -\div(\rho^{2-\lambda}\gd u)+\ev(\nabla u, a_1) +a_0 u,
\end{align}
with $a_1\in \Gamma(\M,T\M)$ and $a_0\in \C^\M$, 
for any $u\in C^\infty(\M)$ and some $\lambda\in\R$. 
We call $\sA$  {\em $(\rho,\lambda)_\R$-regular} if its coefficients satisfy
$$ a_1\in W_\infty^{1,\lambda}(\M; T_\R\M),\quad a_0\in L^\lambda_\infty(\M;\R).$$
Here $T_\R\M$ denotes the real tangent bundle.

Throughout, we assume that the singular data $[\![\rho]\!]$ and the constant $\lambda$ satisfy
\begin{align}
\label{S4: rho & lambda}
\|\rho\|_\infty \leq 1,\quad \lambda> 0. 
\end{align}

\subsection{\bf $\mathscr{R}$-sectorial Operators on singular manifolds}

I first introduce several basic concepts in the study of elliptic differential operators. 
The reader may refer to the treatises \cite{Ama95}, \cite{ArenGrohNage86}, \cite{DenHiePru03} and \cite{PruSim16} for more details of these concepts.

For $\theta\in (0,\pi]$, the open sector with angle $2\theta$ is denoted by
$$\Sigma_\theta:= \{\omega\in \mathbb{C}\setminus \{0\}: |\arg \omega|<\theta \}. $$
\begin{definition}
Let $X$ be a complex Banach space, and $\cA$ be a densely defined closed linear operator in $X$ with dense range. $\cA$ is called sectorial if $\Sigma_\theta \subset \rho(-\cA)$ for some $\theta>0$ and
$$ \sup\{\|\mu(\mu+\cA)^{-1}\| : \mu\in \Sigma_\theta \}<\infty. $$
The class of sectorial operators in $X$ is denoted by $\S(X)$. 
The spectral angle $\phi_\cA$ of $\cA$ is defined by
$$\phi_\cA:=\inf\{\phi:\, \Sigma_{\pi-\phi}\subset \rho(-\cA),\, \sup\limits_{\mu\in \Sigma_{\pi-\phi}} |\mu(\mu+\cA)^{-1}|<\infty \}. $$ 
\end{definition}

Let $\phi\in (0,\pi]$. Define
$$\cHi(\Sigma_\phi):=\{f: \Sigma_\phi \to \mathbb{C}: f \text{ is analytic and } \|f\|_\infty<\infty \} $$
and
$$\cH_0(\Sigma_\phi) =\bigcup_{\alpha,\beta<0} \cH_{\alpha,\beta}(\Sigma_\phi), $$
where 
$$ \cH_{\alpha,\beta}(\Sigma_\phi)=\{f\in \cH(\Sigma_\phi): |f|^\phi_{\alpha,\beta}<\infty \}. $$
Here $|f|^\phi_{\alpha,\beta}=\sup\limits_{|\mu|\leq 1} |\mu^\alpha f(\mu)| + \sup\limits_{|\mu|\geq 1}|\mu^{-\beta} f(\mu) |.$
\begin{definition}
Suppose that $\cA\in \S(X)$. Then $\cA$ is said to admit a bounded $\cHi$-calculus if there are $\phi>\phi_\cA$ and a constant $K_\phi$ such that
\begin{equation}
\label{S4: cHi}
|f(\cA) | \leq K_\phi \|f\|_{\infty} ,\quad f\in \cH_0(\Sigma_\phi).
\end{equation}
The class of such operators is denoted by $\cHi(X)$. The $\cHi$-angle of $\cA$ is defined by 
$$\phi^\infty_\cA:=\inf\{\phi>\phi_\cA: \eqref{S4: cHi} \text{ holds}\}.$$
\end{definition}

\begin{definition}
Let $X$ and $Y$ be two Banach spaces. A family of operators $\mathcal{T}\in \L(X,Y)$ is called $\mathscr{R}$-bounded, if there is a constant $C>0$ and $p\in [1,\infty)$ such that for each $N\in\N$, $T_j\in \mathcal{T}$ and $x_j\in X$ and for all independent, symmetric, $\{-1,1\}$-valued random  variables $\varepsilon_j$ on a probability space $(\Omega,\mathcal{M},\mu)$ the inequality
$$ |\sum\limits_{j=1}^N \varepsilon_j T_j x_j|_{L_p(\Omega; Y)} \leq C|\sum\limits_{j=1}^N \varepsilon_j  x_j|_{L_p(\Omega; X)} $$
is valid. The smallest such $C$ is called $\mathscr{R}$-bound of $\mathcal{T}$. We denote it by $\mathscr{R}(\mathcal{T})$.
\end{definition}

\begin{definition}
Suppose that $\cA\in \S(X)$. Then $\cA$ is called $\mathscr{R}$-sectorial if there exists some $\phi>0$ such that
$$\mathscr{R}_\cA(\phi):=\mathscr{R}\{\mu(\mu+\cA)^{-1}: \mu\in \Sigma_\phi \}<\infty. $$
The $\mathscr{R}$-angle $\phi^R_\cA$ is defined by
$$\phi^R_\cA:=\inf\{ \theta\in (0,\pi): \mathscr{R}_\cA(\pi-\theta) <\infty \}. $$
The class of $\mathscr{R}$-sectorial operators in $X$ is denoted by $\RS(X)$.
\end{definition}

\begin{definition}
A Banach space $X$ said to belong to the class $\mathcal{HT}$ if the Hilbert Transform defined by
$$H(f)(t):=\lim\limits_{\epsilon\to 0} \int_{|s|>\epsilon} f(t-s)\frac{ds}{\pi s} ,\quad t\in\R,\, f\in C_0(\R,X) $$
can be extended to a bounded linear operator on $L_p(\R,X)$ for some $p\in (1,\infty)$.
\end{definition}
It is a  well-known  fact that $L_p(\M)$ is of class $\mathcal{HT}$. This immediately implies that for any $\vartheta\in\R$
\begin{equation}
\label{S4: HT}
L^\vartheta_p(\M) \text{  is of class  }\mathcal{HT},\quad 1<p<\infty.
\end{equation}

Suppose that $X$ is a Banach space of class $\mathcal{HT}$. 
Then by \cite[formula~(2.15), Remark~3.2(1), Theorem~4.5]{DenHiePru03}, 
we obtain
the inclusions
\begin{equation}
\label{S4: all classes}
\cHi(X) \subset   \RS(X) \subset \S(X),
\end{equation}
and the inequalities
$$\phi^\infty_\cA  \geq \phi^R_\cA \geq \phi_\cA. $$

Recall an operator $\cA$ is said to belong to the class $\cH(X_1,X_0)$ for some densely embedded Banach couple $X_1\overset{d}{\hookrightarrow}X_0$, if $-\cA$ generates a strongly continuous analytic semigroup on $X_0$ with $dom(\cA)=X_1$.  $\cA\in \cH(X_1,X_0)$ iff there exists some $\omega\geq 0$ such that
$$\omega+\cA \in \S(X_0) $$
with spectral angle $\phi_\cA <\pi/2$ and $dom(\cA)=X_1$.

Let $X_\R$ be a real Banach lattice with an order $\leq$. See \cite[Chapter~C-I]{ArenGrohNage86}. The complexification of $X_\R$ is a complex Banach lattice defined as follow
\begin{equation}
\label{S4.1: Banach lattice}
X:=X_\R \oplus i X_\R. 
\end{equation}
The positive cone of $X_\R$ is defined by
$$X_\R^+:=\{x\in X_\R:\, 0\leq x\}. $$
\begin{definition}
Let $\vartheta\in\R$, and $X$ be a complex Banach lattice defined as in \eqref{S4.1: Banach lattice}.
Suppose that $\cA\in \S(X)$. Then the semigroup $\{e^{-t\cA}\}_{t\geq 0}$ is real if 
$$e^{-t\cA}X_\R \subset X_\R ,\quad t\geq 0.$$ 
We say that $\{e^{-t\cA}\}_{t\geq 0}$ is positive if 
$$e^{-t\cA}X_\R^+ \subset X_\R^+  ,\quad t\geq 0.$$ 
\end{definition}
For example, taking $X=L_2^\vartheta(\M)$, $\{e^{-t\cA}\}_{t\geq 0}$  real means
$$e^{-t\cA} L_2^\vartheta(\M;\R) \subset L_2^\vartheta(\M;\R) ,\quad t\geq 0,$$ 
and $\{e^{-t\cA}\}_{t\geq 0}$  positive means
$$e^{-t\cA} L_2^\vartheta(\M;\R_+) \subset L_2^\vartheta(\M;\R_+) ,\quad t\geq 0.$$

\begin{prop}
\label{S4: main prop}
Suppose that $(\M,g; \rho)$ is a $C^2$-singular manifold.
Let $\lambda^\prime\in \R$,  and $\lambda>0$.  
Furthermore, assume that the differential operator
\begin{align*}
\cA u:= -\div(\rho^{2-\lambda} \gd u)+\ev(\nabla u, a_1) +a_0 u
\end{align*}
is $(\rho,\lambda)_\R$-regular and satisfies that
$\Rp(\rho^\lambda a_0)$ is so large that there exists some $C_1<2$ and $\omega<0$  such that
\begin{align*}
&\rm{essinf}(\Rp(\rho^{\lambda}a_0)+\omega )>0; \\
&\rho^{\lambda-1} |2\lambda^\prime\rho^{1-\lambda}\gd \rho +a_1|_g \leq C_1 \sqrt{C_{\hat{\sigma}}(\Rp(\rho^{\lambda}a_0) +\omega )};\\
&\rho^{\lambda-1} | (2\lambda^\prime-\lambda)\rho^{1-\lambda}\gd \rho +a_1|_g \leq C_1 \sqrt{C_{\hat{\sigma}}(\Rp(\rho^{\lambda}a_0) +\omega )}. 
\end{align*}
Let $\sA^{\lambda^\prime}_{\cB_0}:= \sA|_{W^{2,\lambda^\prime-\lambda}_{p,\cB_0}(\M)}$.
Then $\sA^{\lambda^\prime}_{\cB_0}$ generates an analytic contraction strongly continuous semigroup on $L_p^{\lambda^\prime}(\M)$ with domain
$W^{2,\lambda^\prime-\lambda}_{p,\cB_0}(\M)$ for all $1<p<\infty$.
\end{prop}
It is proved under the same conditions (with $\sigma=\tau=0$ and $\vec{a}=\rho^{2-\lambda}$) in \cite[Section~3]{Shao15} that 
$\sA^{\lambda^\prime}_{\cB_0}$ generates an analytic contraction strongly continuous semigroup on $L_p^{\lambda^\prime}(\M)$ with domain
$\mathring{W}^{2,\lambda^\prime-\lambda}_p(\M)$
where
$$
\mathring{W}^{2,\lambda^\prime-\lambda}_p(\M_1)=W^{2,\lambda^\prime-\lambda}_{p,\cB_0}(\M_1) \cap W^{2,\lambda^\prime-\lambda}_{p,\cB_1}(\M_1).
$$
However, the characterization of the domain of $\sA$ given in \cite{Shao15} is erroneous. We will fix this error and state the necessary changes herein. 
\begin{proof}
For $\omega$ satisfying the given conditions, let $\Aw:=\sA+\omega \rho^{-\lambda}$ and we associate  with $\sA_1$ a form operator $\aw$ with $D(\aw)=X=W^{1,\lambda^\prime-\lambda/2}_{2,\cB_0}(\M,V)$ defined by 
\begin{align*}
&\aw(u,v)\\
=& \langle \vec{a}\cdot\gd u | \gd v  \rangle_{2,\lambda^\prime}
+ \langle \ev(\nabla u, (2\lambda^\prime+2\tau-2\sigma)\vec{a}\cdot \gd \log \rho +a_1) | v \rangle_{2,\lambda^\prime}
\\
&+\langle (a_0+\omega\rho^{-\lambda}) u | v \rangle_{2,\lambda^\prime} 
\end{align*}
for all $u,v\in X$. Here $\langle u | v \rangle_{2,\lambda^\prime}$ is the inner product in $L^{\lambda^\prime}_2(\M,V)$. 

By the given conditions, $\sA$ satisfies the conditions (A1)-(A2) and (A3'). Following the argument in \cite[Section~3.1]{Shao15}, we can show that $\aw$ is continuous and $X$-coercive, cf. \cite[Proposition~3.2]{Shao15}. This shows that $\aw$ with $D(\aw)=X$ is densely defined, sectorial and closed on $L_2^{\lambda^\prime}(\M)$. By \cite[Theorems~VI.2.1, IX.1.24]{Kato80}, there exists an associated operator $T$ such that $-T$ generates a contractive strongly continuous analytic semigroup on $L_2^{\lambda^\prime}(\M)$, i.e., $\|e^{-tT}\|_{\L(L_2^{\lambda^\prime}(\M))}\leq 1$ for all $t\geq 0$, with domain 
$$D(T):=\{u\in X, \exists ! v\in L_2^{\lambda^\prime}(\M):\aw(u,\phi)=\langle v | \phi \rangle_{2,\lambda^\prime}, \forall \phi\in X \},\quad T u=v, $$
which is a core of $\aw$.
$T$ is unique in the sense that there exists only one operator satisfying 
$$\aw(u,v)= \langle T u, v \rangle_{2,\lambda^\prime},\quad u\in D(T),\, v\in X.$$
On the other hand,  we can infer from \cite[formula~(2.9), Definition~(2.12)]{Shao15} that
$$\langle \Aw u | v \rangle_{2,\lambda^\prime}= \aw(u,v),\quad u,v\in X.$$
So by the uniqueness of $T$, we have
$$\Aw|_{D(T)}=T .$$
Therefore, $-\Aw$ generates a contractive strongly continuous analytic semigroup on $L_2^{\lambda^\prime}(\M)$ with domain $D(\Aw)$:
$$D(\Aw):=\{u\in X, \exists! v\in L_2^{\lambda^\prime}(\M):\aw(u,\phi)=\langle v | \phi \rangle_{2,\lambda^\prime}, \forall \phi\in X \},\quad \Aw u=v. $$
It can be verified that $D(\Aw)$ can be equivalently stated as
$$D(\Aw):=\{u\in X:\, \Aw u\in L_2^{\lambda^\prime}(\M) \}. $$

One can conclude from Propositions~3.1, 3.2 and \cite[Propositions~2.5, 2.9]{Shao15}  that
$$
W^{2,\lambda^\prime-\lambda}_{2,\cB_0}(\M)\subset D(\Aw).
$$
Define 
$$\Bw u:=\rho^\lambda\Aw u.$$
By the previous discussions and Proposition~\ref{S2: change of wgt}, we infer that  
\begin{equation}
\label{S3: Bw1}
\Bw\in \Lis( D(\Aw),L_2^{\lambda^\prime-\lambda}(\M) ).
\end{equation}
On the other hand, in view of \cite[Theorem~5.2]{Ama13b}, by the given conditions and a similar discussion to $\Aw$, we can prove that
$$
\Bw\in \cH(W^{2,\lambda^\prime-\lambda}_{2,\cB_0}(\M),L_2^{\lambda^\prime-\lambda}(\M))\cap \Lis(W^{2,\lambda^\prime-\lambda}_{2,\cB_0}(\M),L_2^{\lambda^\prime-\lambda}(\M)).
$$
Combining with \eqref{S3: Bw1}, we can infer that
$$
D(\Aw)\doteq W^{2,\lambda^\prime-\lambda}_{2,\cB_0}(\M ).
$$
The rest of the proof follows the same arguments as in \cite[Section~3.2]{Shao15}.
\end{proof}
\begin{remark}
Note that the exclusion of $p\in \{3,3/2\}$ in \cite[Theorem~5.2]{Ama13b} is only necessary for maximal regularity theory, cf. Section~4.2, to guarantee $2-2/p\neq i+1/p$ for $i=0,1$. So the generation of analytic semigroup in  \cite[Theorem~5.2]{Ama13b} actually holds for all $1<p<\infty$.
\end{remark}

The following theorem is the theoretical basis of this paper. 
\begin{theorem}
\label{S4: main theorem}
Suppose that $(\M,g; \rho)$ is a $C^2$-singular manifold satisfying $\rho\leq 1$,
$$|\nabla \rho|_g \sim {\bf 1},\quad \|\Delta \rho\|_\infty<\infty$$
on $\M_r:=\{\p\in \M: \rho(\p)<r\}$ for some $r\in (0,1]$. Moreover, assume that the set
$$S_{r_0}:=\{\p\in \M: \rho(\p)=r_0\} $$
is compact and closed for $r_0\in (0,r)$.
Let $\lambda^\prime\in \R$,  and $\lambda\in (0,1)\cup(1,\infty)$.  
Furthermore, assume that the differential operator
\begin{align*}
\sA u:= -\div(\rho^{2-\lambda} \gd u)+\ev(\nabla u, a_1) +a_0 u
\end{align*}
is $(\rho,\lambda)_\R$-regular. Let $\sA^{\lambda^\prime}_{\cB_i}:= \sA|_{W^{2,\lambda^\prime-\lambda}_{p,\cB_i}(\M)}$ for $i=0,1$.
Then there exists some $\omega\geq 0$ such that
$$\omega+\sA^{\lambda^\prime}_{\cB_i} \in \RS( W^{2,\lambda^\prime-\lambda}_{p,\cB_i}(\M),  L^{\lambda^\prime}_p(\M)),\quad 1<p<\infty$$
with $\mathscr{R}$-angle $\phi^R_{\omega+\sA^{\lambda^\prime}_{\cB_i}}<\pi/2$.
\end{theorem}
\begin{proof}
(i) 
For $M>0$, we set 
\begin{align}
\label{S4: h}
h(\p)= M\sg(1-\lambda)\log \rho(\p), \quad \p\in\M. 
\end{align}
A direct computation shows that
$$\rho\gd h =M\sg(1-\lambda) \gd \rho,$$
and
$$
\rho^\lambda \div(\rho^{2-\lambda}\gd h)=M\sg(1-\lambda)\rho \Delta \rho + M|1-\lambda| |\gd \rho|_g^2.  $$
Together with (S3) and  (S4), one can then show that $h \in BC^{2,0}(\M)$,  and
\begin{align}
\label{S4: h on ends}
h\in \sH_\lambda(\M_{r_1},g; \rho) 
\end{align}
with parameters $(c,M)$ for some $r_1\leq r$ sufficiently small. Hence, $(\M_1,g;\rho)$ is a {\em singular manifold} with property $\sH_\lambda$.

(ii) 
Following the proof of \cite[Theorem~5.17]{Shao15}, we introduce a retraction-coretraction system on the weighted Sobolev spaces over $\M$. 
Choose $r_2\in (0,r_1)$, where $r_1$ is defined in \eqref{S4: h on ends}. Define 
$$\M_0:=\{\p\in \M: \rho(\p)\in [r_2,1]\}.$$
Note that $\M_0$ is a {\em uniformly regular Riemannian manifold} with boundary. With a little abuse of notation, we put 
$\M_1:=\M_{r_1}.$

Based on the collar neighborhood theorem, we can find an open neighborhood $U$ of $S_{r_1}$ in $\M_1 \cap \M_0$ such that there exists a diffeomorphism 
$$\phi: U \to S_{r_1} \times [0,1), \quad \phi^* g|_U= g|_{S_{r_1}} + dt^2 ,$$
with $\phi (S_{r_1})=S_{r_1} \times \{0\}$.
We choose functions $\xi, \tilde{\xi} \in BC^\infty([0,1), [0,1])$ such that
$$\xi|_{[0,1/2]}\equiv 1, \quad \xi|_{[3/4,1)} \equiv 0; \quad \tilde{\xi}|_{[0,1/4]}\equiv 0, \quad \tilde{\xi}|_{[1/2,1)} \equiv 1.$$
Set $\hat{\pi}_0:=\phi^*({\bf 1}_{S_{r_1}}\otimes \xi)$ and  $\hat{\pi}_1:=\phi^*({\bf 1}_{S_{r_1}}\otimes \tilde{\xi})$.
We define ,
\begin{align*}
\tilde{\pi}_1(\p) =
\begin{cases}
1, \quad \quad & \p \in \M_1\setminus U,\\
 \hat{\pi}_1, &\p\in U,\\
0, & \text{elsewhere},
\end{cases}
\quad \text{ and } \quad
\tilde{\pi}_0(\p) =
\begin{cases}
1, \quad \quad & \p \in \M\setminus \M_1,\\
 \hat{\pi}_0, &\p\in U,\\
0, & \text{elsewhere}.
\end{cases}
\end{align*}
For $j=0,1$, we set
$$\pi_j = \frac{\tilde{\pi}_j}{\sqrt{ \tilde{\pi}_0^2 + \tilde{\pi}_1^2}}. $$
Then $\{\pi_0^2, \pi_1^2 \}$ forms a partition of unity on $\M$, and $\pi_j\in BC^{\infty,0}(\M).$ 
Define 
$$ \bW_{p,\cB_i}^{s,\vartheta}(\M):=  W_{p,\cB_i}^{s,\vartheta}(\M_0) \times W_{p,\cB_0}^{s,\vartheta}(\M_1).$$ 
It is understood that on $\M_0$, the singularity function can be taken as ${\bf 1}_{\M_0}$, and thus the definition of weighted function spaces on $\M_0$ is independent of the choice of the weight $\vartheta$.
We further introduce two maps:
$$\Lambda^c: W_{p,\cB_i}^{s,\vartheta}(\M)\to  \bW_{p,\cB_i}^{s,\vartheta}(\M):\quad u\mapsto (\pi_0 u, \pi_1 u),$$
and 
$$\Lambda:  \bW_{p,\cB_i}^{s,\vartheta}(\M)\to  W_{p,\cB_i}^{s,\vartheta}(\M):\quad ( u_0, u_1 ) \mapsto  \pi_0 u_0 + \pi_1 u_1.$$
By Proposition~\ref{S2: pointwise multiplication}, we immediately conclude that $\Lambda$ is a retraction from the space $ \bW_{p,\cB_i}^{s,\vartheta}(\M)$ to $W_{p,\cB_i}^{s,\vartheta}(\M)$ with $\Lambda^c$ as a coretraction.

(iii) 
We will prove that $\sA_1:=\sA^{\lambda^\prime}_{\cB_i}|_{\M_1}$ generates an analytic contraction positive strongly continuous semigroup on $L_p^{\lambda^\prime}(\M_1)$ with domain $W^{2,\lambda^\prime-\lambda}_{p,\cB_0}(\M_1)$ for $1<p<\infty$.

Based on Proposition~\ref{S4: main prop}, we can follow the proof of \cite[Theorem~5.18]{Shao15} and immediately prove that
$$\sA_1\in \cH(W^{2,\lambda^\prime-\lambda}_{p,\cB_0}(\M_1), L_p^{\lambda^\prime}(\M_1))$$ 
For this reason, I will only point out necessary modifications to show the contraction property and then present a proof for positivity-preserving herein.  

Let 
$$A_z:=-z \div(\rho^{2-\lambda}\gd h)- z^2\rho^{2-\lambda} |\gd h|_g^2 +z\ev(a_1,\nabla h) +a_0.$$
For any $z=a+ib\in \mathbb{C}$ with $|z|=1$ and $(c,M)$, the parameters of $h$ in \eqref{S4: h on ends}, we quote two conditions from \cite[Section~5.1]{Shao15}.
\begin{itemize}
\item[(H2)] $\displaystyle a\in  (-\frac{1}{2Mc^3}, 0)$, and $|z|=1$.
\item[(H3)] $M$ is sufficiently large such that 
\begin{align*}
\qquad&\rho^\lambda\Rp(A_z)
>C_0 \rho^2|\gd h|^2_g -\omega\\
&\rho^{\lambda-1}|-2z\rho^{2-\lambda}\gd h + t\rho^{2-\lambda}\gd\log\rho +a_1|_g< \frac{2}{\sqrt{C_1}} \sqrt{\rho^\lambda\Rp(A_z)+\omega}
\end{align*}
for some $\omega<0$, $C_0>1$ and $C_1\in (1,C_0)$. Here $t\in \{2\lambda^\prime, 2\lambda^\prime-\lambda\}$.
\end{itemize}

Let $\sA_h= e^{-zh} \circ \sA_1 \circ e^{zh}$.
$(\M_1,g;\rho)$ is a {\em singular manifold} with property $\sH_\lambda$. Based on Proposition~\ref{S4: main prop}, the only changes we need to make to the results in \cite[Section~5]{Shao15} is to change the domain of $\sA_1$ from $\mathring{W}^{2,\lambda^\prime -\lambda}_p(\M_1)$ to $W^{2,\lambda^\prime -\lambda}_{p,\cB_0}(\M_1)$. A modification of \cite[formula~(5.5), Remark~5.22]{Shao15} implies that 
\begin{align}
\label{S4: Ah-semigroup}
\sA_h\in \cH(W^{2,\lambda^\prime -\lambda}_{2,\cB_0}(\M_1), L^{\lambda^\prime}_2(\M_1))\cap \Lis(W^{2,\lambda^\prime -\lambda}_{2,\cB_0}(\M_1), L^{\lambda^\prime}_2(\M_1)) , 
\end{align}
and the semigroup $\{e^{-t\sA_h}\}_{t\geq 0}$ is contractive. 

We have $e^{zh}=\rho^{\sg(1-\lambda)zM}=\rho^{\sg(1-\lambda)aM} \rho^{\sg(1-\lambda) bMi}$.
In Step (iii) of the proof for \cite[Theorem~5.18]{Shao15}, we have proved
$$\rho^{\sg(1-\lambda)bM i}\in BC^{2,0}(\M_1).$$
By Propositions~\ref{S2: pointwise multiplication}, \ref{S2: change of wgt} and the fact that
$e^{zh} e^{-zh}=e^{-zh} e^{zh} ={\bf 1}_{\M_1}$, we infer that
$$[u \mapsto e^{zh}u] \in  \Lis(W^{s,\vartheta}_{p,\cB_0}(\M_1), W^{s,\vartheta+\sg(\lambda-1)aM}_{p,\cB_0}(\M_1)),\quad 1<p<\infty,\quad 0\leq s\leq 2.$$
Note that $|\rho^{\sg(1-\lambda)bM i}|\equiv 1$, in particular, we have
\begin{equation}
\label{S4：e^zh isometry}
[u \mapsto e^{zh}u] \quad \text{is an isometry from } L^{\vartheta}_p(\M_1) \text{ to } L^{\vartheta+\sg(\lambda-1)aM}_p(\M_1).
\end{equation}
A similar argument to \cite[Theorem~5.2]{Shao15} (with proper change from $\mathring{W}^{2,\lambda^\prime -\lambda}_p(\M_1)$ to $W^{2,\lambda^\prime -\lambda}_{p,\cB_0}(\M_1)$) and \eqref{S4: Ah-semigroup} yield
\begin{align*}
\sA_1\in &\cH(W^{2,\lambda^\prime -\lambda+\sg(\lambda-1)aM}_{2,\cB_0}(\M_1),  L^{\lambda^\prime+\sg(\lambda-1)aM}_2(\M_1))\\
&\cap \Lis(W^{2,\lambda^\prime -\lambda+\sg(\lambda-1)aM}_{2,cB_0}(\M_1),  L^{\lambda^\prime+\sg(\lambda-1)aM}_2(\M_1)).
\end{align*}
Since $\lambda^\prime$ is arbitrary and $\sg(\lambda-1)aM\in (-1/2c^3, 1/2c^3)$, it implies that for any $\lambda^\prime\in\R$ 
$$\sA_1\in \cH(W^{2,\lambda^\prime -\lambda}_{2,\cB_0}(\M_1),  L^{\lambda^\prime}_2(\M_1))\cap \Lis(W^{2,\lambda^\prime -\lambda}_{2,\cB_0}(\M_1),  L^{\lambda^\prime}_2(\M_1)) .$$
By \eqref{S4：e^zh isometry}, the semigroup $\{e^{-t\sA_1}\}_{t\geq 0}$ is contractive.

The assertion that $\sA_1$ generates an analytic contraction strongly continuous semigroup on $L_p^{\lambda^\prime}(\M_1)$ with domain $W^{2,\lambda^\prime-\lambda}_{p,\cB_0}(\M_1)$ for $1<p<\infty$ now follows from \eqref{S4：e^zh isometry} and a similar argument to Step (iv) and (v) of the proof for \cite[Theorem~5.18]{Shao15}.

The only thing left to prove is the positivity-preserving of the semigroup $\{e^{-t\sA_1}\}_{t\geq 0}$.
It is easy to see that 
\begin{center}
$u\in W^{2,\lambda^\prime -\lambda}_{2,\cB_0}(\M_1)$ implies that $\bar{u}\in W^{2,\lambda^\prime -\lambda}_{2,\cB_0}(\M_1)$ and $\sA_1 \bar{u}=\overline{\sA_1 u}$. 
\end{center}
Here $\bar{u}$ stands for the complex conjugate of $u$. 
By \cite[Chapter~C-II Remark~3.1]{ArenGrohNage86}, the semigroup $\{e^{-t\sA_1}\}_{t\geq 0}$ is real.

On the other hand, one can verify that $L_p^{\lambda^\prime}(\M_1;\R)$ is a real Banach lattice and $u\in W^{2,\lambda^\prime-\lambda}_{p,\cB_0}(\M_1;\R)$ implies $u^+\in W^{2,\lambda^\prime-\lambda}_{p,\cB_0}(\M_1;\R)$ and by Theorem~\ref{S2: divergence thm}
\begin{align*}
&\quad \langle \sA_1 u^-,u^+ \rangle_{2,\lambda^\prime}
= \langle \rho^{2-\lambda}\gd u^- | \gd u^+  \rangle_{2,\lambda^\prime}\\
&+ \langle 2\lambda^\prime\rho^{2-\lambda} \gd \log \rho +a_1 | u^+ \gd u^- \rangle_{2,\lambda^\prime}
+\langle a_0 u^- | u^+ \rangle_{2,\lambda^\prime}
= 0. 
\end{align*}
We thus infer that
$$\langle \sA_1 u,u^+ \rangle_{2,\lambda^\prime}\geq 0.$$
This implies that $-\sA_1$ is dispersive.
It now follows from \cite[Theorem~2.1]{Phi62} or \cite[Chapter~C-II Theorem~1.2]{ArenGrohNage86} that 
the semigroup $\{e^{-t\sA_1}\}_{t\geq 0}$ is positive.

(iv) In this step, we will show that $\sA_1\in \cHi(L^{\lambda^\prime}_p(\M_1))$ with $\phi^\infty_{\sA_1}<\pi/2$ for all $1<p<\infty$. 

Given any $\psi\in L_1(\R_+)$ with compact support, we have
$$m(\sA_1):=\int_0^\infty e^{-t\sA_1}\psi(t)\, dt\in \L(L^{\lambda^\prime}_p(\M_1)). $$
By the discussion in Step~(iii), it is not hard to verify that the semigroup $\{\rho^{\lambda^\prime}\circ e^{-t\sA_1}\circ \rho^{-\lambda^\prime}\}_{t\geq 0}$ is positive contraction strongly continuous on $L_p(\M_1)$. We immediately have  
$$\hat{m}(\sA_1):=\int_0^\infty \rho^{\lambda^\prime}\circ e^{-t\sA_1}\circ \rho^{-\lambda^\prime}\psi(t)\, dt\in \L(L_p(\M_1)). $$
It was established in \cite[Theorem~3.5]{CoifWei78} that
$$\|\hat{m}(\sA_1)\|_{\L(L_p(\M_1))} \leq \|\psi^*\|,$$
where $\|\psi^*\|$ is the norm of the convolution operator $[u\mapsto u \ast \psi]$ on $L_p(\R)$. This gives rise to the following estimate. For any $u\in L^{\lambda^\prime}_p(\M_1)$,
\begin{align*}
\| m(\sA_1) u\|_{L^{\lambda^\prime}_p(\M_1)}&= (\int_{\M_1}  |\rho^{\lambda^\prime}\int_0^\infty e^{-t\sA_1} u \psi(t)\, dt|^p\, d\mu_g )^{1/p} \\
&= (\int_{\M_1}  |\hat{m}(\sA_1) \rho^{\lambda^\prime} u |^p\, d\mu_g )^{1/p}=\|\hat{m}(\sA_1) \rho^{\lambda^\prime} u\|_{L_p(\M_1)}\\
&\leq \|\psi^*\| \|\rho^{\lambda^\prime} u\|_{L_p(\M_1)}=\|\psi^*\| \|u\|_{L^{\lambda^\prime}_p(\M_1)}.
\end{align*}
Therefore, we obtain
\begin{equation}
\label{S4: transference}
\| m(\sA_1)\|_{\L(L^{\lambda^\prime}_p(\M_1))} \leq \|\psi^*\|.
\end{equation}
Based on \eqref{S4: transference}, we can follow the proof of \cite[Theorem~2]{Duong89} and prove that
\begin{equation}
\label{S3.1: Lp cHi}
\sA_1\in \cHi(L^{\lambda^\prime}_p(\M_1)),\quad\phi^\infty_{\sA_1}=\pi/2,\quad 1<p<\infty.
\end{equation}
When $p=2$,  we infer from \cite[Theorem~8]{McI86}  that
\begin{equation}
\label{S3.1: L2 cHi}
\sA_1\in \cHi(L^{\lambda^\prime}_2(\M_1)),\quad \phi^\infty_{\sA_1}<\pi/2.
\end{equation}
We denote this angle by $\phi_2$.

We can follow the ideas in \cite{Cow83} to prove $\phi^\infty_{\sA_1}<\pi/2$ for all $1<p<\infty$.
Given $0< \phi <\varphi < \frac{\pi}{2}<\beta$, where $\phi\in (\phi_2,\frac{\pi}{2})$ and $\beta>\frac{\pi}{2}$. There exists $\theta\in (0,1)$ such that
$$ \varphi=(1-\theta)\phi_2 +\theta \frac{\pi}{2}. $$
For $\psi\in(0,\pi)$,  let 
$$S_\psi:=\{x+iy\in \mathbb{C}: y\in (-\psi,\psi), x\in\R\} .$$
For any $f\in \cHi(\Sigma_\varphi)$, set $n(z):= f(e^z)$, where $z\in \Sigma_\varphi$. Then $n$ is analytic and bounded in $S_\varphi$.
By \cite[Lemma~1]{Cow83}, there exists a family $\{n_\zeta: \zeta=\xi+i\eta: \xi\in (0,\pi), \eta\in \R \}$ of functions satisfying the following properties:
\begin{itemize}
\item[(a)] $n_\zeta$ is analytic and bounded  in $S_{\xi-\varepsilon}$ for any small $\varepsilon>0$ and 
$$\sup\{ |n_\zeta(z)|:z\in S_{\xi-\varepsilon} \}\leq C \|n\|_\infty \varepsilon^{-1/2} $$
for some fixed $C>0$.
\item[(b)] $n_\varphi=n$.
\item[(c)] For any $x\in\R$, $[\zeta\mapsto n_\zeta(x)]$ is analytic.
\end{itemize}

Now we define
$$f_\zeta(z)=n_\zeta(\log(z)),\quad \zeta=\xi+i\eta,\quad z\in \Sigma_\xi. $$
Then $f_\zeta \in \cHi(\Sigma_{\xi-\varepsilon})$ for any small $\varepsilon>0$ and $f_\varphi=f$.

We first look at $\zeta$ with ${\Rp}\zeta=\phi$. Then for any $\phi_1\in (\phi_2,\phi)$, we have $f_\zeta\in \cHi(\Sigma_{\phi_1})$. 
By \eqref{S3.1: L2 cHi}, it holds
$$\|f_\zeta(\sA_1)\|_{\L(L^{\lambda^\prime}_2(\M))}\leq C_0\|f_\zeta\|_\infty\leq C C_0 \|f\|_\infty(\phi-\phi_1)^{-1/2}=:B_0. $$

When ${\Rp}\zeta=\beta$, by \eqref{S3.1: Lp cHi} and a similar argument as above
$$\|f_\zeta(\sA_1)\|_{\L(L^{\lambda^\prime}_q(\M))}\leq B_1 $$
for some fixed constant $B_1$.

We put $m_\zeta:=\rho^{\lambda^\prime}f_\zeta(\sA_1) \rho^{-\lambda^\prime}$. Then 
$$\|m_{\phi+i\eta}\|_{\L(L_2(\M))}\leq B_0, \quad \|m_{\beta+i\eta}\|_{\L(L_q(\M))}\leq B_1. $$
Now for $p$ satisfying $\frac{1}{p}=\frac{1-\theta}{2} +\frac{\theta}{q}$, applying the Stein interpolation theorem, we obtain
$$\|m_\varphi\|_{\L(L_p(\M))}\leq B_0^{1-\theta}B_1^\theta.$$
This implies that
$$\|f(\sA_1)\|_{\L(L^{\lambda^\prime}_p(\M))}= \|f_\varphi(\sA_1)\|_{\L(L^{\lambda^\prime}_p(\M))}\leq B_0^{1-\theta}B_1^\theta. $$ 
Since the choice of $q$ is arbitrary, we have proved that
$$\sA_1\in \cHi(L^{\lambda^\prime}_p(\M_1)),\quad\phi^\infty_{\sA_1}<\pi/2,\quad 1<p<\infty.$$

(v) 
By \eqref{S4: HT} and \eqref{S4: all classes}, it holds that 
\begin{equation}
\label{S4: A1-RS}
\sA_1 \in \RS(L^{\lambda^\prime}_p(\M_1)),\quad 1<p<\infty
\end{equation}
with $\phi^R_{\sA_1}<\pi/2$

On the other hand, $(\M_0,g)$ is uniformly regular, and thus $\rho|_{\M_0}\sim {\bf 1}_{\M_0}$. 
Therefore, the choice of weights is immaterial in the definition of function spaces on $(\M_1,g)$.  $\sA_0:=\sA|_{\M_0}$ is uniformly strongly $\rho$-elliptic in the sense of \cite[formula~(5.1)]{Ama13b}.
It is proved in \cite[Theorem~5.2]{Ama13b} that  
\begin{equation} 
\label{S4: A0-RS}
\omega_0+\sA_0 \in \RS(L_p^{\lambda^\prime}(\M_0)),\quad 1<p<\infty
\end{equation}
for some $\omega_0 \geq 0$
with $\phi^R_{\sA_0}<\pi/2$ and 
domain $W^{2,\lambda^\prime-\lambda}_{p,\cB_i}(\M_0)$ for $i=0,1$.

(vi) 
An easy computation shows that for $j=0,1$ and any $v\in \mathcal{D}(\M)$
\begin{align}
\label{S4: commutator}
\notag \pi_j \sA v=& \sA_j (\pi_j v)  + 2\rho^{2-\lambda}\ev(\nabla v, \gd \pi_j) +  [\div(\rho^{2-\lambda}\gd \pi_j)- \ev(\nabla \pi_j, a_1)] v\\
=&:\sA_j (\pi_j v) +\cB_j v .
\end{align}
Note that 
$$\gd \pi_j \in BC^{\infty,\vartheta}(\M, T\M)$$ 
for any $\vartheta\in\R$.
Based on this observation and  Propositions~\ref{S2: pointwise multiplication}, \ref{S2: nabla}, \ref{S2: div-tangent},  and \cite[Corollaries~7.2, 12.2]{Ama13}, we infer that
\begin{align}
\label{S5.3: Besov-Lp}
\cB_j \in \L(W_{p,\cB_i}^{1+r, \lambda^\prime-\lambda(1+r)/2}(\M), L_p^{\lambda^\prime}(\M_j)), \quad i,j=0,1.
\end{align}
for any $r\in (0,1)$.
Put $\bar{\sA}:=(\sA_j)_{j=0}^1$ and
$$E_1:= \bW^{2,\lambda^\prime-\lambda}_{p,\cB_i}(\M),\quad E_0:=\boldsymbol{L}_p^{\lambda^\prime}(\M),\quad E_\gamma:=(E_0,E_1)_{\gamma,p}. $$
By the definition and convexity of $\mathscr{R}$-boundedness, see \cite[Proposition~3.8]{DenHiePru03}, it is clear that \eqref{S4: A1-RS} and \eqref{S4: A0-RS} imply
\begin{equation}
\label{S4: bA-RS}
\omega_0+\bar{\sA}\in \RS(E_0)	,\quad 1<p<\infty
\end{equation}
for some $\omega_0\geq 0$
with $\phi^R_{\omega_0+\bar{\sA}}<\pi/2$.  
By \cite[Proposition~3.4, Theorem~4.4]{DenHiePru03}, for some $\theta\in (\pi/2, \pi-\phi^R_{\omega_0+\bar{\sA}})$
$$\mathscr{R}\{(\omega_0+\bar{\sA})(\mu+\omega_0+\bar{\sA})^{-1}: \mu\in \Sigma_\theta \}<\infty .$$
By \eqref{S5.3: Besov-Lp}, we have 
$$\cB:=(\cB_j)_{j=0}^1 \in \L(W_{p,\cB_i}^{1+r, \lambda^\prime-\lambda(1+r)/2}(\M), E_0).$$
One can conclude that
$$\cB\Lambda \in \L(E_{\frac{1+r}{2}}, E_0).$$
Interpolation theory thus implies that 
for every $\varepsilon>0$ there exists some positive constant $C(\varepsilon)$ such that for all $\bu=(u_j)_{j=0}^1 \in E_1$
\begin{equation}
\label{S4: cB}
\|\cB\Lambda \bu\|_{E_0} \leq \varepsilon \|\bu\|_{E_1} + C(\varepsilon)\|\bu\|_{E_0} .
\end{equation}
It is well-known that \eqref{S4: bA-RS} implies that there is a constant $\cE>0$ such that 
$$|\omega|^{1-j} \| (\omega+\bar{\sA})^{-1}\|_{\L(E_0,E_j)}\leq \cE,\quad j=0,1,\quad \omega\geq \omega_0. $$
Together with \eqref{S4: cB}, it implies
$$\|\cB\Lambda \bu\|_{E_0} \leq \varepsilon \cE \|(\omega_0+\bar{\sA})\bu\|_{E_0} + C(\varepsilon)\|\bu\|_{E_0} .$$
By \cite[Proposition~4.3, Theorem~4.4]{DenHiePru03}, we can find some $\omega_1 \geq \omega_0$ such that 
\begin{equation}
\label{S4: bA2-RS-1}
\omega_1+\bar{\sA}+\cB\Lambda \in \RS(E_0)	,\quad 1<p<\infty
\end{equation}
with $\phi^R_{\omega_1+\bar{\sA}+\cB\Lambda}<\pi/2$, and for some $\theta>\pi/2$
\begin{equation}
\label{S4: bA2-RS-2}
\mathscr{R}\{\mu (\mu+\omega_1+\bar{\sA}+\cB\Lambda )^{-1}: \mu\in \Sigma_\theta \}<\infty .
\end{equation} 
\eqref{S4: commutator} shows that
$$\Lambda^c (\mu+\omega_1+\sA)u =(\mu+\omega_1+\bar{\sA})\Lambda^c u +\cB\Lambda\Lambda^c u = (\mu +\omega_1+\bar{\sA}+\cB\Lambda)\Lambda^c u.$$
One readily checks that
\begin{align*}
&\Lambda(\mu +\omega_1+\bar{\sA}+\cB\Lambda)^{-1}\Lambda^c (\mu+\omega_1+\sA)\\
=&\Lambda(\mu +\omega_1+\bar{\sA}+\cB\Lambda)^{-1} (\mu +\omega_1+\bar{\sA}+\cB\Lambda)\Lambda^c
=\id_{W^{2,\lambda^\prime-\lambda}_{p,\cB_i}(\M)}. 
\end{align*}

(vii) On the other hand, one can also view 
$$\cB_j\in \L(W_{p,\cB_i}^{1+r, \lambda^\prime-\lambda(1+r)/2}(\M_j),L_p^{\lambda^\prime}(\M)).$$
Let $\mathfrak{B}\bu:=\sum_{j=0}^1 \cB_j u_j$ for $\bu=(u_j)_{j=0}^1 \in E_{\frac{1+r}{2}}$.
Following a similar argument as in Step (vi), we have for some $\omega_2\geq\omega_1$
\begin{align*}
&(\mu+\omega_2+\sA)\Lambda(\mu+\omega_2+\bar{\sA}-\Lambda^c\fB)^{-1}\Lambda^c\\
=&\Lambda(\mu+\omega_2+\bar{\sA}-\Lambda^c\fB)(\mu+\omega_2+\bar{\sA}-\Lambda^c\fB)^{-1}\Lambda^c=\id_{L_p^{\lambda^\prime}(\M)} 
\end{align*}
and  for some $\theta>\pi/2$
\begin{equation}
\label{S4: bA3-RS}
\mathscr{R}\{\mu (\mu+\omega_2+\bar{\sA}-\Lambda^c\fB )^{-1}: \mu\in \Sigma_\theta \}<\infty .
\end{equation} 
Therefore, we have
$$(\mu+\omega_2+\sA)^{-1}=\Lambda(\mu+\omega_2+\bar{\sA}-\Lambda^c\fB)^{-1} \Lambda^c $$
and by \cite[Proposition~3.4]{DenHiePru03}
\begin{align*}
\mathscr{R}\{\mu (\mu+\omega_2+\sA)^{-1}: \mu\in \Sigma_\theta \}<\infty 
\end{align*}
The asserted statement now is an easy consequence of \cite[Theorem~4.4]{DenHiePru03}.
\end{proof}

We say $u,v \in C^k(\M;\R)$ are $C^k$-equivalent, which is denoted by $u\sim_k v$, if
$$u\sim v,\quad |\nabla ^i u |_g \sim | \nabla^i v|_g,\quad i=1,\cdots,k. $$

Given any compact submanifold $\Sigma\subset (\M,g)$, the distance function 
is a well-defined smooth function in a collar neighborhood $\mathscr{U}_\Sigma$ of $\Sigma$. The distance ball at $\Sigma$ with radius $r$ is defined by 
$$\B_\M(\Sigma,r):= \{\p\in \M: {\rm dist}_\M(\p,\Sigma)<r \}. $$
\begin{definition}
\label{S4.1: Torn mfd}
Suppose that $(\sM,g)$ is an $m$-dimensional uniformly regular Riemannian manifold.
\begin{itemize}
\item[(i)] $(\sM,g)$ is an $m$-dimensional uniformly regular Riemannian manifold, and $\boldsymbol{\Sigma}=\{\Sigma_j:j=1,\cdots,k\}$ is a finite set of disjoint $(m-1)$-dimensional compact submanifolds  such that $\Sigma_j\subset \mathring{\sM}$ or $\Sigma_j \subset \partial\sM$. Put $G_0:=\sM\setminus \cup_{j=1}^k \Sigma_j$ and 
$$\mathscr{B}_{j,r}:= \bar{\B}_\sM(\Sigma_j,r)\cap G_0,\quad j=1,\cdots,k.$$ 
Furthermore, the singularity function $\rho$ satisfies 
\begin{align}
\label{S5.3: near bdry}
\rho \sim_2 {\rm dist}_{\sM}(\cdot, \Sigma_j)\quad \text{in } \mathscr{B}_{j,r} 
\end{align}
for some $r\in (0,\delta)$, where $\delta< {\rm diam}(\sM)$ and $\mathscr{B}_{i,\delta} \cap \mathscr{B}_{j,\delta}=\emptyset$ for $i\neq j$, and
$$\rho\sim {\bf 1},\quad \text{elsewhere on }G_0. $$
\item[(i')] Assume that $\boldsymbol{\Sigma}:=\{\p_j:j=1,\cdots,k\}$ is a family of points  in $\mathring{\sM}$. In addition, $G_0$ and $\rho$ are defined as in (i).
\item[(ii)] ${\bf G}=\{G_1,\cdots,G_n\}$ is a finite set of disjoint $m$-dimensional closed uniformly mild wedges.
More precisely, there is a diffeomorphism $f_i: G_i \to W(R_i,B_i,\Gamma_i)$ with $R_i\in \mathscr{C_U}(J_0)$.
Let $I_r:=(0,r]$ and
$$\mathscr{G}_{i,r}:=f_i^{-1} (\phi_P(I_r\times B_i)\times\Gamma_i),\quad i=1,\cdots,n.$$
Moreover, the singularity function $\rho$ satisfies 
\begin{align}
\label{S5.3: on wedge end}
\rho \sim_2 f_i^*(\phi_P^*(R_i|_{I_r} \otimes {\bf 1}_{B_i})\otimes {\bf 1}_{\Gamma_i}) \quad \text{in } \mathscr{G}_{j,r}
\end{align}
for some $r\in (0,1]$, and
$$\rho\sim {\bf 1},\quad \text{elsewhere on }G_i. $$
\item[(iii)] $\{G_0\}\cup {\bf G}$ forms a covering for $\M$, and $ G_0\cap G_i \subset \partial G_0 \cap \partial G_i$.
\end{itemize}
If $(\M,g;\rho)$ satisfies condition (i)-(iii), then it is called a {\em torn singular manifold  with  closed uniformly mild wedge ends}. If $(\M,g;\rho)$ satisfies (i'), (ii) and (iii), then it is called a {\em punctured singular manifold with closed uniformly mild wedge ends}.
Analogues of these two concepts were first introduced in \cite{Shao15} and then slightly modified here.
\end{definition}
One can show that
 \eqref{S5.3: near bdry} and \eqref{S5.3: on wedge end} imply that
\begin{align}
\label{S5.3: equiv near ends}
|\Delta \rho|<\infty \quad \text{in }\mathscr{B}_{j,r} \text{ and } \mathscr{G}_{j,r}.
\end{align}

\begin{theorem}
\label{S4: main theorem-2}
Suppose that $\lambda\in [0,1)\cup(1,\infty)$ when $(\M,g; \rho)$ is a torn $C^2$-singular manifold with closed uniformly mild wedge ends, or $\lambda\in [0,m)\cup(m,\infty)$ when $(\M,g; \rho)$ is an $m$-dimensional punctured $C^2$-singular manifold with closed uniformly mild wedge ends. 
Let $\lambda^\prime\in \R$. 
Furthermore, assume that the differential operator
\begin{align*}
\sA u:= -\div(\rho^{2-\lambda} \gd u)+\ev(\nabla u, a_1) +a_0 u
\end{align*}
is $(\rho,\lambda)_\R$-regular.
Then there exists some $\omega\geq 0$ such that $\sA^{\lambda^\prime}_{\cB_i}:=\sA|_{W^{2,\lambda^\prime -\lambda}_{p,\cB_i}(\M)}$ satisfies
$$\omega+\sA^{\lambda^\prime}_{\cB_i}\in \RS( L^{\lambda^\prime}_p(\M)),\quad 1<p<\infty,$$
with $\mathscr{R}$-angle $\phi^R_{\omega+\sA^{\lambda^\prime}_{\cB_i}}<\pi/2$.
\end{theorem}
\begin{proof}
This follows from \cite[Proposition~5.10, Corollary~5.20, Remarks~5.16, 5.21, 5.22]{Shao15} and a similar argument as in Theorem~\ref{S4: main theorem}.
\end{proof}



\subsection{\bf  Maximal $L_p$-regularity of  second order elliptic differential operators}

In order to study nonlinear parabolic equations on {\em singular manifolds}, we will establish maximal $L_p$-regularity for differential operators of the form \eqref{S4: operator A}.
To this end, we consider the following abstract Cauchy problem 
\begin{equation}
\label{S4: Cauchy problem}
\left\{\begin{aligned}
\partial_t u(t) +\cA u(t) &=f(t), &&t\geq 0\\
u(0)&=0 . &&
\end{aligned}\right. 
\end{equation}
\begin{definition}
Assume that $X_1\overset{d}{\hookrightarrow}X_0$ is some densely embedded Banach couple.
Suppose that $\cA\in \S(X_0)$ with $dom(\cA)=X_1$.
The Cauchy problem \eqref{S4: Cauchy problem} has maximal $L_p$-regularity if for any 
$$f\in L_p(\R_+; X_0) ,$$
\eqref{S4: Cauchy problem} has a unique solution
$$u\in L_p(\R_+; X_1) \cap H^1_p(\R_+; X_0) .$$
We denote this by 
$$\cA\in \mathcal{MR}_p(X_0).$$
\end{definition}

\begin{theorem}
\label{S4.2: main thm-MR}
Let $\lambda^\prime\in \R$ and $p\neq 3/2,3$.
Suppose that a $C^2$-singular manifold $(\M,g; \rho)$  and $\lambda$  satisfy the conditions in Theorem~\ref{S4: main theorem} or \ref{S4: main theorem-2}.
Furthermore, assume that the differential operator
\begin{align*}
\sA u:= -\div(\rho^{2-\lambda} \gd u)+\ev(\nabla u, a_1) +a_0 u 
\end{align*}
is $(\rho,\lambda)_\R$-regular. 
Let $J=[0,T]$ with $T\in (0,\infty)$ and $i=0,1$.
Then for every 
$$(f,u_0) \in L_p(J; L_p^{\lambda^\prime}(\M))  \times W^{2-\frac{2}{p}, \lambda^\prime-\lambda(1-\frac{1}{p})}_{p,\cB_i}(\M),$$
and
$$
h\in L_p(J; W^{2-i-\frac{1}{p},\lambda^\prime-\lambda+i+\frac{1}{p}}_p(\partial\M))\cap W^{1-\frac{i}{2}-\frac{1}{2p}}_p(J; L_p^{ \lambda^\prime+i+\frac{1}{p}-\lambda(\frac{i}{2}+\frac{1}{2p})}(\M))
$$  
satisfying the compatibility condition
\begin{align*}
\begin{cases}
\cB_0 u_0=h(0) \text{ on }\partial\M \quad &    \text{if }i=0 \text{ and }3/2<p,\\
\cB_1 u_0 = h(0) \text{ on }\partial\M   &    \text{if }i=1 \text{ and }3<p,\\
\end{cases}
\end{align*}
the Cauchy problem  
\begin{equation}
\label{S4.2 Cauchy pb-2}
\left\{\begin{aligned}
\partial_t u(t) +\sA u(t) &=f(t)  &&\text{on}&&\M_T;\\
\cB_i u &=h &&\text{on}&&\partial\M_T;\\
u(0)&=u_0  &&\text{on}&&\M.&&
\end{aligned}\right.
\end{equation}
has a unique solution
$$u \in  L_p(J; W^{2 , \lambda^\prime-\lambda }_{p,\cB_i}(\M)) \cap H^1_p(J; L_p^{\lambda^\prime}(\M)) .$$
\end{theorem}
\begin{proof}
Theorems~\ref{S4: main theorem}, \ref{S4: main theorem-2} and \cite[Theorem~4.4]{DenHiePru03} show that, for some $\omega\geq 0$, $\sA^{\lambda^\prime}_{\cB_i}:= \sA|_{W^{2,\lambda^\prime-\lambda}_{p,\cB_i}(\M)}$ satisfies 
\begin{equation}
\label{S4.2: Aw-MR}
\omega+\sA^{\lambda^\prime}_{\cB_i}\in \mathcal{MR}_p( L_p^{\lambda^\prime}(\M)).
\end{equation}
By Proposition~\ref{S2: interpolation}, we have
$$(L_p^{\lambda^\prime}(\M), W^{2 , \lambda^\prime-\lambda}_{p,\cB_i}(\M))_{1-1/p,p}=W^{2-\frac{2}{p}, \lambda^\prime-\lambda(1-\frac{1}{p})}_{p,\cB_i}(\M).$$
We set 
$$\bE_0(J):=L_p(J; L_p^{\lambda^\prime}(\M)),$$
and
$$\bE_1(J):= L_p(J; W^{2, \lambda^\prime-\lambda}_{p,\cB_i}(\M)) \cap H^1_p(J; L_p^{\lambda^\prime}(\M)); $$
and put
$$
\bbF_0(J):=L_p(J; L_p^{\lambda^\prime-\lambda+i+1/p}(\partial\M)),
$$
and
$$
\bbF_1(J):=L_p(J; W_p^{2, \lambda^\prime-\lambda+i+1/p}(\partial\M)) \cap H^1_p(J; L_p^{\lambda^\prime +i+1/p}(\partial\M)),
$$
and 
$$
\bE_\theta(J)=(\bE_0(J),\bE_1(J))_{\theta,p},\quad
\bbF_\theta(J)=(\bbF_0(J),\bbF_1(J))_{\theta,p}.
$$
\cite[Definition~(8.3), Theorem~20.3]{AmaAr} show that for $\theta=1-\frac{i}{2}-\frac{1}{2p}$
$$
\cB_i\in \L(\bE_1(J), \bbF_\theta(J))
$$
is a retraction and there exists a coretraction $\cB^c_i\in \L(\bbF_\theta(J), \bE_1(J))$ such that 
$
\cB_i \circ \cB^c_i =\id.
$
Note that in \cite{AmaAr} the author uses the following notations:
$$
\bE_1(J)=W^{2/\vec{2}, (\lambda^\prime-\lambda, \lambda)}_{p.\cB_i}(\M)
$$
and
$$
\bbF_1(J)=W^{0/\vec{2}, (\lambda^\prime-\lambda+i+1/p, \lambda)}_{p.\cB_i}(\partial\M),\quad \bbF_1(J)=W^{2/\vec{2}, (\lambda^\prime-\lambda+i+1/p, \lambda)}_{p.\cB_i}(\partial\M).
$$
To obtain a precise characterization for  $\bbF_\theta(J)$, we apply  interpolation theory and infer that
\begin{align*}
&(L_p(J;L_p^{\lambda^\prime-\lambda+i+1/p}(\partial\M)),L_p(J;W_p^{2, \lambda^\prime-\lambda+i+1/p}(\partial\M)) )_{\theta,p}\\
&\doteq L_p(J; W_p^{2\theta, \lambda^\prime-\lambda+i+1/p}(\partial\M)).
\end{align*} 
On the other hand, following the proof for \cite[Theorem~21.1]{AmaAr}, we can show that
\begin{align*}
&(L_p(J;L_p^{\lambda^\prime-\lambda+i+1/p}(\partial\M)),H^1_p(J;L_p^{ \lambda^\prime+i+1/p}(\partial\M)) )_{\theta,p}\\
&\doteq W^\theta_p(J; L_p^{ \lambda^\prime+i+\frac{1}{p}-\lambda(1-\theta)}(\partial\M)).
\end{align*} 
Therefore,
$$
\bbF_\theta(J)\doteq
L_p(J; W_p^{2\theta, \lambda^\prime-\lambda+i+\frac{1}{p}}(\partial\M))\cap W^\theta_p(J; L_p^{ \lambda^\prime+i+\frac{1}{p}-\lambda(1-\theta)}(\partial\M))
$$
is the boundary trace space for $\bE_1(J)$. Let $u^*=\cB^c_i (h)$. Then $u=u^*+\hat{u}$ solves \eqref{S4.2 Cauchy pb-2}, where $\hat{u}$ is a solution to
\begin{equation}
\label{S4.2 Cauchy pb-3}
\left\{\begin{aligned}
\partial_t u +\sA u &=f-\partial_t u^* -\sA u^*=:f^*  &&\text{on}&&\M_T;\\
\cB_i u &=0 &&\text{on}&&\partial\M_T;\\
u(0)&=u_0 - u^*(0)=:\hat{u}_0  &&\text{on}&&\M.&&
\end{aligned}\right.
\end{equation}
We will show \eqref{S4.2 Cauchy pb-3} has a unique solution in $\bE_1(J)$.

Given  Banach spaces $E$ and $\mathbb{E}$ such that 
$\mathbb{E}\hookrightarrow L_{1,loc}(J;E)$, define
$$e^{\omega t}\mathbb{E}:=\{u\in L_{1,loc}(J;E): e^{-\omega t}u\in \mathbb{E}\}, $$
equipped with the norm
$$\|u\|_{e^{\omega t}\bE}:= \| e^{-\omega t} u\|_{\bE}.$$
The boundedness of $J$ gives  
\begin{equation}
\label{S4.2: equi}
e^{\omega t}\bE_j(J)\doteq \bE_j(J),\quad j=0,1.
\end{equation}
One readily checks that $w(t)=e^{-t\sA^{\lambda^\prime}_{\cB_i}}\hat{u}_0$ solves
$$\partial_t u(t) +\sA u(t) =0,\quad \cB_i u=0  \quad u(0)=\hat{u}_0.$$
\cite[formula~(III.1.5.6), Proposition~4.10.3]{Ama95} and \eqref{S4.2: equi} imply that
$ w\in \bE_1(J).$

Hence it only remains to prove  
that for any $f \in \bE_0(J)$, the Cauchy problem
$$\partial_t v(t) +(\omega+\sA) v(t)  = e^{-\omega t}f(t),  \quad v(0)=0$$
has a unique solution $v\in \bE_1(J)$, or equivalently,
$$u=e^{\omega t}v \in e^{\omega t}\bE_1(J) $$
uniquely solves \eqref{S4.2 Cauchy pb-3} with $\hat{u}_0=0$.
This is an immediate consequence of \eqref{S4.2: Aw-MR}.

Let $\gamma: \bE_1(J) \to E:=W^{2-\frac{2}{p}, \lambda^\prime-\lambda(1-\frac{1}{p})}_{p,\cB_i}(\M): u\mapsto u(0)$. Then
\begin{equation}
\label{S4: spatial trace}
\gamma\in \L(\bE_1(J),E).
\end{equation}
By the maximal regularity property, cf. \cite[Formula~(1.2)]{Pru03} and \eqref{S4: spatial trace}, we have
\begin{align*}
\|u\|_{\bE_1(J)} &\leq \|u^*\|_{\bE_1(J)} + \|\hat{u}\|_{\bE_1(J)} \\
&\leq C(\|h\|_{\bbF_\theta(J)} + \| f^* \|_{\bE_0(J)} + \|\hat{u}_0\|_E)\\
&\leq  C(\|h\|_{\bbF_\theta(J)} + \| f\|_{\bE_0(J)} + \|\partial_t u^* +\sA u^* \|_{\bE_0(J)} + \|u_0\|_E + \|u^*(0)\|_E)\\
&\leq C(\|h\|_{\bbF_\theta(J)} + \| f\|_{\bE_0(J)} +\|u_0\|_E).
\end{align*}
This establishes the uniqueness of the solution $u=u^*+\hat{u}$.
\end{proof}

\begin{cor}
\label{S4.2: main cor-MR}
Under the same conditions as in Theorem~\ref{S4.2: main thm-MR}, 
if
$(f,h,u_0)$  is real valued, so is the solution $u$.
\end{cor}
\begin{proof}
If $h$ is real, then $u^*=\cB^c_i (h)$ is real.
Since the semigroup $\{e^{-t\sA}\}_{t\geq 0}$ is real on $L_p^{\lambda^\prime}(\M)$,  the variation of parameters formula for the solution $\hat{u}$ to \eqref{S4.2 Cauchy pb-3}, i.e.,
$$\hat{u}(t)= e^{-t\sA}(u_0-u^*(0)) +\int_0^t e^{-(t-s)\sA} f^*(s)\, dt, \quad t\geq 0$$
immediately gives the desired assertion.
\end{proof}


\section{\bf The Yamabe flow}

Suppose that $(\M,g_0;\rho)$ is a $C^2$-singular manifold of dimension $m$ for $m\geq 3$. The Yamabe flow reads as
\begin{equation}
\label{S5: Yamabe flow eq1}
\left\{\begin{aligned}
\partial_t g&=-R_g g  &&\text{on}&&\M_T;\\
H_g&=0  &&\text{on}&&\partial\M_T;\\
g(0)&=g_0  &&\text{on}&&\M,&&
\end{aligned}\right.
\end{equation}
where $R_g$ is the scalar curvature with respect to the metric $g$, and $H_g$ is the mean curvature of $g_{\partial\M}$. The evolving metric $g$ is in the conformal class of the  background metric $g_0$, i.e., $[g_0]$. Note that a {\em singular manifold} must be incomplete. Therefore, this flow actually starts with an incomplete metric.

\subsection{\bf Wellposedness of the Yamabe flow on singular manifolds}

In this subsection, we apply Theorem~\ref{S4.2: main thm-MR} and Corollary~\ref{S4.2: main cor-MR} to establish the general theory for existence and uniqueness of solutions to the Yamabe flow \eqref{S5: Yamabe flow eq1} on {\em singular manifolds} satisfying conditions in Theorem~\ref{S4: main theorem} or \ref{S4: main theorem-2}

Let $c(m):=\frac{m-2}{4(m-1)}$. The conformal Laplacian operator $L_g$ with respect to the metric $g$ is defined  by
$$
L_g u:=\Delta_g u -c(m)R_g u.
$$
Here $\Delta_g $ is the Laplace-Beltrami operator with respect to $g$ defined by 
$$\Delta_g:=\div_g\circ\gd_g. $$
Let $g=u^{\frac{4}{m-2}}g_0$ for $u>0$. By rescaling the time variable, equation \eqref{S5: Yamabe flow eq1} is equivalent to 
\begin{equation*}
\left\{\begin{aligned}
\partial_t u^{\frac{m+2}{m-2}}&=\frac{m+2}{m-2}L_{g_0} u  &&\text{on}&&\M_T; \\
\ev(\nabla u, \boldsymbol{n}_0)&=0  &&\text{on}&&\partial\M_T; \\
u(0)&={\bf 1}_\M  &&\text{on}&&\M,&&
\end{aligned}\right.
\end{equation*}
where $\boldsymbol{n}_0$ is the inward pointing normal of $g_0$ on $\partial\M$, cf.  \cite{Brend02}. It is equivalent to solving the following equation:
\begin{equation}
\label{S5: Yamabe flow eq2}
\left\{\begin{aligned}
\partial_t u &=u^{-\frac{4}{m-2}}L_{g_0} u  &&\text{on}&&\M_T; \\
\ev(\nabla u, \boldsymbol{n}_0)&=0  &&\text{on}&&\partial\M_T; \\
u(0)&={\bf 1}_\M  &&\text{on}&&\M. &&
\end{aligned}\right.
\end{equation}
Put 
$$P(u)v:=u^{-\frac{4}{m-2}}\Delta_{g_0} v ,\quad Q(u):=-c(m) u^{\frac{m-6}{m-2}}R_{g_0} .$$
We set
$$E_0:= L_p^{2-\frac{m+2}{p}}(\M;\R),\quad E_1:=W^{2, -\frac{m+2}{p}}_{p,\cB_1}(\M;\R).$$
Then it follows from Proposition~\ref{S2: interpolation} that
$$E_{1-1/p}:=(E_0,E_1)_{1-1/p,p}= W^{2-\frac{2}{p}, -\frac{m}{p}}_{p,\cB_1}(\M;\R).$$
By Theorem~\ref{S4.2: main thm-MR} and Corollary~\ref{S4.2: main cor-MR}, the operator
\begin{equation}
\label{S5.1: YF-MR}
P(u_0)=\Delta_{g_0} \in \mathcal{MR}_p( E_0 ),\quad 1<p<\infty
\end{equation}
with domain $E_1$. 
We assume that
\begin{equation*}
p>m+2.
\end{equation*}
Proposition~\ref{S2: Sobolev embedding} then implies
$$E_{1-1/p} \hookrightarrow BC^{1,0}(\M;\R). $$

We first look at the linear problem
\begin{equation}
\label{S5: Yamabe flow-linear}
\left\{\begin{aligned}
\partial_t u &=\Delta_{g_0} u  &&\text{on}&&\M_T; \\
\ev(\nabla u, \boldsymbol{n}_0)&=0  &&\text{on}&&\partial\M_T; \\
u(0)&={\bf 1}_\M  &&\text{on}&&\M. &&
\end{aligned}\right.
\end{equation}
It is clear that $u^*\equiv {\bf 1}_\M$ solves \eqref{S5: Yamabe flow-linear}.

Then we look at the nonlinear problem
\begin{equation}
\label{S5: Yamabe flow-nonlinear}
\left\{\begin{aligned}
\partial_t u -P(u+u^*)u &=Q(u+u^*)  &&\text{on}&&\M_T; \\
\ev(\nabla u, \boldsymbol{n}_0)&=0  &&\text{on}&&\partial\M_T; \\
u(0)&={\bf 0}_\M  &&\text{on}&&\M. &&
\end{aligned}\right.
\end{equation}

We take 
$$\hat{U}:=\B_R({\bf 0}_\M):=\{v\in BC^{1,0}(\M;\R): \|v\|_{1,\infty;0}<R\}$$ 
to be an open ball around ${\bf 0}_\M$ in $BC^{1,0}(\M;\R)$ with $R<1$, and set $U:=i^{-1}(\hat{U})$, where 
$i$ is the embedding of  $E_{1-1/p}$ into $BC^{1,0}(\M;\R)$. Then $U$ is open in $E_{1-1/p}$  and 
$$\inf(v+{\bf 1}_\M)\geq 1-R>0,\quad v\in U .$$
On the $C^2$-{\em uniformly regular Riemannian manifold} $(\M,\hat{g}):=(\M,g/\rho^2;{\bf 1}_\M)$,
we can define unweighted $BC^k$-spaces as in Section~2.2 with respect to the metric $\hat{g}$ and by taking $\rho={\bf 1}_\M$. We denote these spaces by $BC^k(\hat{\M})$. 
By \cite[formula~(4.4)]{Ama13b}, it holds that
$$ BC^k(\hat{\M})\doteq BC^{k,0}(\M)$$
An observation into the proof for \cite[Proposition~6.3]{ShaoSim13} shows that this proposition remains true for $BC^k(\hat{\M};\R)$. Therefore, it holds for any $\alpha\in\R$ that
\begin{equation*}
[v\mapsto v^\alpha] \in C^\omega(\hat{U} + {\bf 1}_\M, BC^{1,0}(\M;\R)),
\end{equation*}
where $\omega$ is the symbol of real analyticity, and 
$$\hat{U} + {\bf 1}_\M := \{u \in L_{1,loc}(\M;\R): u=u_1+{\bf 1}_\M \quad \text{with } u_1\in \hat{U}; \, \|\cdot\|_{1,\infty;0}\}$$
is open in $BC^{1,0}(\M;\R)$. 
Moreover, $[v\mapsto (v+{\bf 1}_\M)]\in C^\omega(U, \hat{U} + {\bf 1}_\M)$. Therefore,
\begin{equation}
\label{S5.1: Nemyskij}
[v\mapsto (v+{\bf 1}_\M)^\alpha] \in C^\omega(U, BC^{1,0}(\M;\R)),
\end{equation}
In view of Propositions~\ref{S2: nabla} and \ref{S2: div-tangent}, we have
$$\Delta_{g_0} \in \L(E_1, E_0). $$
Now one infers from Proposition~\ref{S2: pointwise multiplication} and \eqref{S5.1: Nemyskij} that
\begin{equation}
\label{S5.1: reg of P}
[v\mapsto P(v +u^*)] \in C^\omega(U, \L(E_1,E_0))
\end{equation}

Our most important assumption for the Yamabe flow can be stated as follows.
\begin{itemize}
\item[(YF)] There exists some $\varepsilon\in (\frac{2}{p},\infty)$ such that
\begin{align}
\label{S5.1: ASP}
\rho^{\varepsilon-\frac{m+2}{p}}\in L_p(\M;\R),\quad R_{g_0}\in BC^{\infty, 2-\varepsilon}(\M;\R). 
\end{align}
\end{itemize}
We would like to comment on the second entry of \eqref{S5.1: ASP}. This condition is not optimal. Indeed, in \cite[formula~(4.13)]{Shao14}, it is shown that
$$R_{g_0}\in BC^{\infty, 2}(\M;\R).$$
So (YF) actually asks $R_{g_0}$ to be less singular.
However, we will show in Section~5.2 that (YF) is satisfied by plenty of examples of {\em singular manifolds}.

As a consequence of \eqref{S5.1: Nemyskij} and (YF), we conclude from Proposition~\ref{S2: pointwise multiplication} that
$$[v\mapsto Q(v +u^*)] \in C^\omega(U, BC^{1,2-\varepsilon}(\M;\R))$$
and for any $v\in BC^{1,2-\varepsilon}(\M;\R)$
\begin{align*}
\| v\|_{p; 2-\frac{m+2}{p}}  \leq \|v\|_{\infty; 2-\varepsilon} \| \rho^{\varepsilon-\frac{m+2}{p}}\|_p \leq C \|v\|_{\infty; 2-\varepsilon}.
\end{align*}
Combining the above two results, we infer that
\begin{equation}
\label{S5.1: reg of Q}
[v\mapsto Q(v +u^*)] \in C^\omega(U, E_0).
\end{equation}
\cite[Theorem~2.1]{CleLi93}, \eqref{S5.1: YF-MR}, \eqref{S5.1: reg of P} and \eqref{S5.1: reg of Q} imply that there exists a unique solution 
$$\tilde{u}\in \bE_1(J):= L_p(J; E_1) \cap H^1_p(J; E_0) $$
solving \eqref{S5: Yamabe flow-nonlinear}. 
Now it is clear that $\hat{u}:=\tilde{u} +u^*$ solves \eqref{S5: Yamabe flow eq2}. We will show that $\hat{u}$ is indeed the unique solution in the class $\bE_1(J)\oplus \R_\M$, where 
$\R_\M$ is the space of all real-valued constant functions on $\M$. First we have 
$$\bE_1(J) \hookrightarrow C(J; E_{1-1/p}), $$
and $\R_\M \cap E_{1-1/p} = \{{\bf 0}_\M\}$. If $u\in \bE_1(J)\oplus \R_\M$  solves \eqref{S5: Yamabe flow eq2}, then we have a unique decomposition of the solution $u=u_1+u_2$ with $u_1\in \bE_1(J)$ and $u_2\in \R_\M$. Since $u(0)=u_1(0) + u_2= {\bf 1}_\M$, we infer that $u_2={\bf 1}_\M$. Thus $u_1\in \bE_1(J)$ is a solution to \eqref{S5: Yamabe flow-nonlinear} by our construction, and the uniqueness follows immediately.

Now we are in a position to state the following well-posedness theorem for the Yamabe flow.
\begin{theorem}
\label{S5.1: YF-thm}
Suppose that $(\M,g_0;\rho)$ is a $C^2$-singular manifold satisfying conditions in Theorem~\ref{S4: main theorem} or \ref{S4: main theorem-2}, and $p>m+2$. 
We assume that condition (YF) is fulfilled. 
Then the Yamabe flow~\eqref{S5: Yamabe flow eq2} has a unique positive solution
$$u \in  L_p(J; W^{2, -\frac{m+2}{p}}_{p,\cB_1}(\M,\R)) \cap H^1_p(J; L_p^{2-\frac{m+2}{p}}(\M;\R))\oplus \R_\M . $$
for some $J:=[0,T]$ with $T>0$. Moreover,
$$u\in C(J; W^{2-\frac{2}{p}, -\frac{m}{p}}_{p,\cB_1}(\M;\R))\oplus \R_\M. $$
\end{theorem}
The local positivity of the solution $u$ in the above theorem follows from the fact that $u\in C(J; W^{2-\frac{2}{p}, -\frac{m}{p}}_p(\M;\R))\oplus {\bf 1}_\M $ and $W^{2-\frac{2}{p}, -\frac{m}{p}}_p(\M;\R)\hookrightarrow BC^{1,0}(\M;\R)$.
\begin{remark}
We observe that our solution to the Yamabe flow~\eqref{S5: Yamabe flow eq1} remains in the set $U \oplus {\bf 1}_\M$ at least for a short period of time $J=[0,T]$. For any fixed $t_0\in J$, we have 
$$0<\inf u(t_0)\leq \sup u(t_0) <\infty.$$
This means that $(\M,g(t);\rho)$ is a {\em singular manifold} for all $t\in J$. So the unique solution obtained in Theorem~\ref{S5.1: YF-thm} preserves the type of singularities and is not instantaneously complete. This has the same feature of the solution obtained in \cite{BahVer14} for compact manifold with an incomplete edge singularity.
\end{remark}


\subsection{\bf Examples of the Yamabe flow starting on incomplete manifolds}

\begin{prop}
\label{S5.2: YF-prop}
Suppose that $(\M,g_0;\rho)$ is a torn or punctured $C^2$-singular manifold  with ${V\!ol}_{g_0}(\M)<\infty$. Then condition (YF) is satisfied for all $\varepsilon\in (\frac{m+1}{p},2]$. 
\end{prop}
\begin{proof}
By Definition~\ref{S4.1: Torn mfd}, $(\M,g_0;\rho)$ is constructed by removing a finite family of $(m-1)$-dimensional compact submanifolds  $\boldsymbol{\Sigma}:=\{\Sigma_j:j=1,\cdots,k\}$ or a finite family of points  $\boldsymbol{\p}:=\{\p_j:j=1,\cdots,k\}$ from a {\em uniformly regular Riemannian manifold}  $(\sM, \tilde{g})$ with $\tilde{g}|_\M =g_0$. 

The proof of \cite[Theorem~1.2(b)]{DisShaoSim} also applies to the scalar curvature $R_{\tilde{g}}$. This implies that the covariant derivatives of $R_{g_0}=R_{\tilde{g}}|_\M$ are uniformly bounded, i.e., $R_{g_0}\in BC^{\infty,0}(\M)$. Hence, by taking $\varepsilon\leq 2$, we infer from Proposition~\ref{S2: embedding} that
$$R_{g_0}\in BC^{\infty,2-\varepsilon}(\M).$$
Set $\mathscr{B}_{j,r}:= \B_\sM(\Sigma_j,r)\cap \M$ with $j=1,\cdots,k$. 
By the collar neighborhood theorem, there exists an open neighborhood $\mathscr{V}_{j,\varepsilon}$ of $\partial \Sigma_j$ in  $\sM$, and a diffeomorphism $f_j$ such that 
$$f_j:  \mathscr{V}_{j,\varepsilon} \to \partial \Sigma_j \times (-\varepsilon,\varepsilon), \quad (f_j)_* \tilde{g}|_{\mathscr{V}_{j,\varepsilon}}= \tilde{g}|_{\partial \Sigma_j} + dt^2 , $$
for some $\varepsilon>0$. 
Let $T^\perp \partial \Sigma_j$ denote the normal bundle of $\partial \Sigma_j$ in $\sM$. At every point $\p\in \partial\Sigma_j$, there exists a unique $\nu_\p\in T^\perp_\p \partial \Sigma_j$ such that
$$ T_\p f_j \nu_\p=e_1\in T_0 \R. $$
Then, $f_j^{-1}(\p,t)=exp_\p(t\nu_\p)$, where $exp_\p$ is the exponential map at $\p$. 
Therefore, letting $\dot{J}_\varepsilon= (-\varepsilon,\varepsilon)\setminus \{0\}$, we have
$$(f_j)_* \rho (\p,t):= t \beta_j(\p),\quad  \text{in }  \partial\Sigma_j \times \dot{J}_\varepsilon,$$
for some $\beta_j \in C^\infty (\partial\Sigma_j)$ and $\beta_j \sim {\bf 1}_{\partial\Sigma_j}$. 
Because of the compactness of $\partial\Sigma_j$, it is easy to compute that
\begin{align*}
\| \rho^{\varepsilon-\frac{m+2}{p}}\|_p^p=\int\limits_{\partial\Sigma_j\times \dot{J}_\varepsilon} |t\beta_j(\p) |^{\varepsilon p -m-2}\, (f_j)_* d\mu_{\tilde{g}}\leq  C\int\limits_{\partial\Sigma_j\times \dot{J}_\varepsilon} |t |^{\varepsilon p -m-2} \,  dt \,d\mu_{\tilde{g}|\partial\Sigma_j}.
\end{align*}
Therefore, $\varepsilon>\frac{m+1}{p}$ will guarantee 
$$\rho^{\varepsilon-\frac{m+2}{p}}\in L_p(\M;\R).$$
The case of {\em punctured singular manifolds} follows by considering the polar coordinates around each $\p_j$.
\end{proof}

\begin{example}
\label{S5.2: YF-eg}
The following are examples of $C^2$-singular manifolds satisfying the conditions in Proposition~\ref{S5.2: YF-prop}.
\begin{itemize}
\item[(i)] Assume that $\Omega \subset \R^m$ is a bounded $C^3$-domain, $g_m$ is the Euclidean metric in $\R^m$, and $\d(\cdot):={\sf dist}(\cdot, \partial\Omega)$ near  $\partial\Omega$. The fact that $(\Omega, g_m; \d)$ is a $C^2$-singular manifold follows from \cite[Theorem~1.6]{Ama14} or \cite[Proposition~3.7]{Shao14}. 

By \cite[Corollary~4.3]{Ama14}, $(\bar{\Omega}, g_m)$ is a  uniformly regular Riemannian manifold.  Definition~\ref{S4.1: Torn mfd} implies that $(\Omega, g_m; \d)$ is a  $C^2$-torn singular manifold satisfying the conditions in Proposition~\ref{S5.2: YF-prop}, and thus the Yamabe flow~\eqref{S5: Yamabe flow eq1} admits a unique local solution on the incomplete manifold $(\M,g_0;\rho)=(\Omega, g_m; \d)$.
 
\item[(ii)] More generally, suppose that $(\sM,\tilde{g})$ is a $C^2$ compact manifold with boundary; $\partial_0\sM \sqcup \partial_1\sM =\partial\sM$ and $\partial_i\sM$ are both open. Then by Definition~\ref{S4.1: Torn mfd}, $(M,g_0):=(\sM\setminus \partial_0\sM, \tilde{g}|_\M)$ is a $C^2$-torn singular  manifold satisfying the conditions in Proposition~\ref{S5.2: YF-prop}. So Theorem~\ref{S5.1: YF-thm} holds true in this case.
In contrast to the situation in \cite{Brend02}, we do not require any boundary condition on $\partial_0\sM$.
\end{itemize}
\end{example}

\section*{Acknowledgements}
The author  would like to record his debt of gratitude to Prof. Jan Pr\"uss for his inspiring and valuable suggestions.
Thanks also go to Prof. Herbert Amann for helpful discussions throughout this project and my thesis advisor, Prof. Gieri Simonett, for his guidance.

\end{document}